\documentclass[]{amsart}
\usepackage{amsmath, amsthm, amssymb}
\usepackage[T1]{fontenc}
\usepackage[utf8]{inputenc}     
\usepackage[dvipsnames]{xcolor}
\usepackage{xstring}            
\usepackage{enumerate}
\usepackage[
  colorlinks=true, linkcolor=blue, linkbordercolor=blue, citecolor=red, citebordercolor=red, urlcolor=blue, linktocpage=true
]{hyperref}
\usepackage[capitalise, nameinlink, noabbrev, nosort]{cleveref} 
\usepackage[lite]{amsrefs}  

\usepackage{mathtools}          
\usepackage{bm}           
\usepackage[shortcuts]{extdash}       
\usepackage{subcaption}         
\usepackage{accents}
\usepackage{xspace}

\usepackage{xparse} 

\usepackage{tikz-cd}
\usetikzlibrary{decorations.markings}
\tikzset{double line with arrow/.style args={#1,#2}{decorate,decoration={markings,%
mark=at position 0 with {\coordinate (ta-base-1) at (0,-2pt);
\coordinate (ta-base-2) at (0,2pt);},
mark=at position 1 with {\draw[#1] (ta-base-1) -- (0,-2pt);
\draw[#2] (ta-base-2) -- (0,2pt);
}}}}

\DeclarePairedDelimiter\abs{\lvert}{\rvert}        

\DeclarePairedDelimiter\braces{\{}{\}}            

\newcommand{\bigbrace}[1]{\braces[\big]{#1}}    
    
\newcommand{\bigmid}{\mathrel{\big|}}            

\DeclareMathOperator{\diam}{diam}            

\theoremstyle{plain}
\newtheorem{thm}{Theorem}[section]        
\newtheorem{prop}[thm]{Proposition}        
\newtheorem{lem}[thm]{Lemma}            \newtheorem{lemma}[thm]{Lemma}
\newtheorem{cor}[thm]{Corollary}        \newtheorem{corollary}[thm]{Corollary}
            \newtheorem{question}[thm]{Question}

\newtheorem{fact}[thm]{Fact}

\newtheorem*{thm*}{Theorem}            \newtheorem*{theorem*}{Theorem}
\newtheorem*{prop*}{Proposition}        \newtheorem*{proposition*}{Proposition}
\newtheorem*{lem*}{Lemma}            \newtheorem*{lemma*}{Lemma}
\newtheorem*{cor*}{Corollary}            \newtheorem*{corollary*}{Corollary}
\newtheorem*{qu*}{Question}            \newtheorem*{question*}{Question}
\newtheorem*{conj*}{Conjecture}            \newtheorem*{conjecture*}{Question}
\newtheorem*{prob*}{Problem}        \newtheorem*{problem*}{Problem}
\newtheorem*{fact*}{Fact}
\newtheorem*{claim*}{Claim}
\newtheorem*{case*}{Case}

\numberwithin{equation}{section}

\newtheorem{alphthm}{Theorem}            

\newtheorem{alphcor}[alphthm]{Corollary}           

\theoremstyle{definition}

        \newtheorem{convention}[thm]{Convention}

\newtheorem*{de*}{Definition}            \newtheorem{definition*}{Definition}
\newtheorem*{notation*}{Notation}
\newtheorem*{conv*}{Convention}            \newtheorem*{convention*}{Convention}

\theoremstyle{remark}
            \newtheorem{remark}[thm]{Remark}

\newtheorem*{rmk*}{Remark}        \newtheorem*{remark*}{Remark}

\DeclareMathAlphabet{\mathbit}{OT1}{cmr}{bx}{it}

\crefname{thm}{Theorem}{Theorems}              \crefname{theorem}{Theorem}{Theorems}
\crefname{prop}{Proposition}{Propositions}     \crefname{proposition}{Proposition}{Propositions}
\crefname{lem}{Lemma}{Lemmas}                  \crefname{lemma}{Lemma}{Lemmas}
\crefname{rmk}{Remark}{Remarks}                \crefname{remark}{Remark}{Remarks}
\crefname{cor}{Corollary}{Corollaries}         \crefname{corollary}{Corollary}{Corollaries}
\crefname{qu}{Question}{Questions}             \crefname{question}{Question}{Questions}
\crefname{conj}{Conjecture}{Conjectures}       \crefname{conjecture}{Conjecture}{Conjectures}
\crefname{prob}{Problem}{Problems}             \crefname{problem}{Problem}{Problems}
\crefname{fact}{Fact}{Facts}
\crefname{claim}{Claim}{Claims}
\crefname{case}{Case}{Cases}
\crefname{alphthm}{Theorem}{Theorems}          \crefname{alphcor}{Corollary}{Corollaries}
\crefname{alphprop}{Proposition}{Propositions}
\crefname{alphprolem}{Problem}{Problems}

\newcommand{\Cat}[1]{\ifmmode \text{\normalfont \textbf{#1}} \else {\normalfont \textbf{#1}}\fi}

\newcommand\MakeComments[3]{
  \newcounter{#2comment}\addtocounter{#2comment}{1}%
  \newcommand{#1}[1]{%
     \textbf{\color{#3}(\StrChar{#2}{1}\arabic{#2comment})}%
     \marginpar{\tiny\raggedright\textbf{\color{#3}(\StrChar{#2}{1}\arabic{#2comment}) #2:} ##1}%
    \addtocounter{#2comment}{1}%
  }%
}

\definecolor{newpurple}{rgb}{0.8, 0, 0.9}

\newcommand{\Debug}{1}

\ifnum \Debug = 0 \usepackage[notref,notcite]{showkeys}
\fi

\MakeComments{\anote}{Agelos}{Orange}
\MakeComments{\fnote}{Fede}{ForestGreen}

\newcommand{\cd}{\ensuremath{\mathcal D}}
\newcommand{\cv}{\ensuremath{\mathcal V}}
\newcommand{\R}{\ensuremath{\mathbb R}}
\newcommand{\Z}{\ensuremath{\mathbb Z}}

\newcommand{\labtequ}[2]{
 \begin{equation} \label{#1} 	\begin{minipage}[c]{0.9\textwidth}  #2 \end{minipage} \ignorespacesafterend \end{equation} }

\usepackage{tikz-cd, tikz} 
\usepackage{transparent}
\usepackage{xcolor}
\usetikzlibrary{decorations.pathmorphing}
\usetikzlibrary{hobby} 
\usetikzlibrary{patterns} 
\usetikzlibrary{intersections} 
\tikzset{
bnode/.style={circle,fill=black, inner sep =0.8 pt},
sbnode/.style={circle,fill=black, inner sep =0.6pt},
middlearrow/.style={
        decoration={markings,
            mark= at position 0.55 with {\arrow{#1}} ,
        },
        postaction={decorate}
    }
}
\def \globalscale {2}

\title{Triangulating surfaces quasi-isometrically}

\author{Agelos Georgakopoulos}        
\address[Agelos Georgakopoulos]{{Mathematics Institute}\\ {University of Warwick}\\  {CV4 7AL, UK.}}  

\author{Federico Vigolo}       
\address[Federico Vigolo]{Mathematisches Institut, Georg-August-Universit\"{a}t G\"{o}ttingen, Bunsenstr. 3-5, 37073 G\"{o}ttingen, Germany.}

\thanks{AG is supported by EPSRC grant  EP/V009044/1.}

\begin{document}
	
\begin{abstract}
  We prove that if a complete Riemannian surface $(\Sigma,d_\Sigma)$ is quasi-isometric to some bounded degree graph $G$, then $\Sigma$ admits a triangulation whose 1-skeleton is quasi-isometric to it when equipped with the simplicial metric. We study several variants of the problem, and identify the right condition making it an if and only if statement. 
\end{abstract}

\maketitle



\section{Introduction}

A classical result of Rado states that every closed topological surface admits a triangulation, and it is well-known that this extends to the non-compact case \cite{AhlSarRie}. If a surface $\Sigma$ has additional structure, such as a metric or conformal structure, finding triangulations of $\Sigma$ with desirable properties can have far-reaching consequences about $\Sigma$. For instance, there are connections between equilateral triangulations and conformal structures \cite{BisRemNon}; thick triangulations and quasimeromorphic mappings \cite{saucan2005note}; triangulations with bounded edge-length and large-scale geometry/geometric group theory \cite{Maillot},\cite{kahn2012counting}*{Lemma 2.2}.

In this paper we focus on the large-scale geometric structure of surfaces equipped with a complete Riemannian metric. Our main result is the following.

\begin{alphthm} \label{thm: uni net intro}
If complete Riemannian surface $(\Sigma, d_\Sigma)$ has a uniform net, then $\Sigma$ admits a triangulation of whose (simplicial) $1$-skeleton is quasi-isometric to $(\Sigma, d_\Sigma)$.
\end{alphthm}

Admitting a uniform net (see \Cref{subsec QI} for definitions) is well-known to be equivalent to being quasi-isometric to a bounded-degree graph \cite{KanaiRough2}.

By allowing degenerate triangles in our `triangulation' we can ensure that the $0$-skeleton (i.e.\ the vertices) of the triangulation constructed in \cref{thm: uni net intro} coincides with any given uniform net of $\Sigma$. 
Moreover, the 1-skeleton of our triangulation will have bounded vertex degrees up to ignoring parallel edges. The latter property in fact yields a converse of \Cref{thm: uni net intro}, which we make precise in \Cref{cor uni net}.

Similar statements had been proved by Maillot \cite{Maillot}, and our techniques provide simpler proofs and strengthenings of said results (see \Cref{thm Maillot} and the discussion following it).
  
\medskip

Our proof of \cref{thm: uni net intro} relies on the construction of triangulations that are quasi-isometric with respect to the induced length metric. To make this precise,
let us first recall that any rectifiable, arc-connected, subspace $X$ of a Riemannian manifold $(\Sigma,d_\Sigma)$ inherits a \emph{length metric $d^X_\ell$} via
\labtequ{dell}{$d^X_\ell(x,y)\coloneqq \inf\{ \ell(\gamma) \mid \gamma \text{ is a $x$--$y$~arc in } X \},$}
whereby $\ell(\gamma)$ denotes the length of $\gamma$ with respect to $d_\Sigma$. 
{The following statement is of independent interest.}

\begin{alphthm} \label{thm:triangulationIntro}
Let $(\Sigma,d_\Sigma)$ be a complete Riemannian surface, and $\Xi\in \R_{\geq 0}$. Then there is a triangulation $\mathcal T$ of $(\Sigma,d_\Sigma)$ each 2-cell $C$ of which has diameter at most $\Xi$ with respect to its  length metric $d^C_\ell$. Moreover, the identity map from the 1-skeleton $(G,d^G_\ell)$ of $\mathcal T$ to $(\Sigma,d_\Sigma)$ is ($1$-Lipschitz and) a quasi-isometry.
\end{alphthm}

\cref{thm:triangulationIntro} was recently proved by Ntalampekos \& Romney \cite{ntalampekos2023polyhedral} using fine geometric machinery, but we provide an alternative, more combinatorial proof. Our proof further shows that, whenever $(\Sigma,d_\Sigma)$ is equipped with a quasi-isometrically embedded graph $G$ that is fine enough to approximate homotopy types of curves in a ``metrically controlled'' manner, then we can turn $G$ into a quasi-isometric triangulation of $\Sigma$ without adding extra vertices to it. We refer to \cref{thm: filling triangulation} for the precise technical statement we prove.

In contrast, the idea of the proof of \cite{ntalampekos2023polyhedral} is to start with a sufficiently fine triangulation of $\Sigma$ so that all the triangles are convex. Once this is done, one may further refine this triangulation in such a way that the resulting triangulation is $\epsilon$-isometric to $\Sigma$ (\emph{i.e.}\ it is quasi-isometric via a $(1,\epsilon)$-quasi-isometry), see \cite{ntalampekos2023polyhedral}*{Proposition 5.2}. 

\medskip

The lengths of the edges of the graph $G$ of \cref{thm:triangulationIntro} obey a uniform upper-bound $c \Xi$, but no uniform lower bound. In particular, \cref{thm:triangulationIntro} does \emph{not} imply that the embedding of $G$ in $\Sigma$ is a quasi-isometry when $G$ is equipped with its simplicial graph-metric, and hence does not directly imply \cref{thm: uni net intro}.

Understanding whether metric graphs with desirable properties are quasi-isometric to simplicial graphs with analogous properties is generally an interesting and hard problem. For instance, the first author \& Papasoglu \cite{GeoPapMin} asked whether for every graph $G$ embedded into a Riemannian plane $(\Sigma, d_\Sigma)$, the metric $(G,d_\ell)$ is quasi-isometric to that of a  simplicial planar graph. An equivalent formulation is whether every planar metric graph is quasi-isometric to a simplicial planar graph. This question was answered in the affirmative by Davies \cite{DavStr}. 

In turn, it was natural to ask whether every Riemannian surface can be triangulated in such a way that the embedding of the $1$-skeleton is a quasi-isometry when the latter is given its simplicial graph metric. It turns out that the answer is no in general: Davies \cite{DavSur} constructed complete Riemannian surfaces which are not quasi-isometric 
to any embedded \emph{simplicial graph}. That is, he constructed a surface $(\Sigma,d_\Sigma)$ so that there does not exist a graph with all edges having length 1, with an embedding into $\Sigma$ that is also a quasi-isometry with $(\Sigma,d_\Sigma)$. 
This shows that \cref{thm: uni net intro,thm:triangulationIntro} are  sharp.

\subsection{Generalising the setup}

Although we have been focusing on unbounded surfaces, our results and proofs can be applied essentially verbatim to families of compact surfaces with increasing diameter and genus, providing triangulations with uniform distortion bounds. On the other hand, we do not know if our results extend to higher-dimensional manifolds, see \Cref{ssec: hi dim} for more.

{We prove our results} for surfaces with boundary (with slightly more technical formulations): the precise statements we prove are \cref{thm: filling triangulation,thm: triangulation,thm Maillot}, which directly imply \cref{thm:triangulationIntro} and \cref{thm: uni net intro}.
Moreover, almost all the techniques we use do not rely on the fact that the metric is Riemannian, so they apply to topological surfaces equipped with complete length metrics. The one important exception is \cref{lem: resolve intersections}, which relies on a perturbation argument that is not available for arbitrary length surfaces.

\subsection{Eliminating the multiplicative or additive distortion}

Having obtained quasi-isometric triangulation theorems, it is natural to ask if we need to have both the multiplicative and the additive distortion involved in our quasi-isometries. 
Questions of this type currently attract interest in the context of coarse graph theory.
Nguyen, Scott \& Seymour \cite{NgScSeAsyII} observe that there are planar (simplicial) graphs $G,H$ such that there is a $(M,A)$-quasi-isometry from $G$ to $H$ for some (multiplicative, resp.\ additive) constants $M,A$, but no $(1,A')$-quasi-isometry.
It is thus natural to ask whether \cref{thm:triangulationIntro} can be sharpened by imposing that $(G,d^G_\ell)$ be $(1,A)$-quasi-isometric to $(\Sigma,d_\Sigma)$. Ntalampekos \& Romney \cite{ntalampekos2023polyhedral} show that this is the case, and James Davies (private communication) pointed out that our construction can be refined to achieve this too:

\begin{alphcor}\label{cor M=1 intro}
  The triangulation of \cref{thm:triangulationIntro} can be chosen so that $(G,d^G_\ell)$ is $(1,A)$-quasi-isometric to $(\Sigma,d_\Sigma)$.
\end{alphcor}

\begin{remark}
  Even if $\Sigma$ has a uniform net, the construction of \cref{cor M=1 intro} will always contain arbitrarily short edges.
  We expect that it is not true in general that one can find a triangulation whose simplicial $1$-skeleton is $(1,A)$-quasi-isometric to $(\Sigma,d_\Sigma)$.
\end{remark}

Similarly, one can ask whether the additive constant $A$ can be set to 0 by modifying the quasi-isometry between fixed planar graphs $G,H$, i.e.\ whether $G,H$ are bi-Lipschitz equivalent. In Section~\ref{sec bLe} we combine our results with a theorem of Burago \& Kleiner \cite{BuKlSep} to provide a strong negative answer: 

\begin{alphcor} \label{cor AD intro}
There are plane graphs $G,H$ (with bounded degrees and face-boundary sizes, potentially with parallel edges) which are quasi-isometric but not bi-Lipschitz equivalent.
\end{alphcor}

An interesting open question is whether every (non-planar) graph $G$ which is quasi-isometric to a planar graph $H$ must be $(1,A)$-quasi-isometric to some planar graph $H'$; see  \cite{NgScSeAsyII} for more.

\subsection{On mapping class groups and flip graphs.}

This work also provides an interesting perspective on `asymptotic' mapping class groups and flip graphs.
The mapping class group $\mathrm{MSC}(\Sigma, X)$ of a marked surface $\Sigma, X$ is a central topic in geometric topology. Computing it directly is generally hard, and therefore it is often studied by constructing a graph $G$ whose group of graph-automorphisms is naturally isomorphic to $\mathrm{MSC}(\Sigma, X)$. One such graph is the \emph{flip-graph} $\mathcal F$ of $(\Sigma, X)$, the vertices of which are certain triangulations of $\Sigma$. The fruitful relation between $\mathrm{MSC}(\Sigma, X)$ and $Aut(\mathcal F)$ is known to break down when $(\Sigma, X)$ is of infinite type. As we explain in \cref{ssec: flip-graph}, our triangulations can be used to adapt the definitions of the flip-graph (and the mapping class group) so that the aforementioned relation may be preserved.

\subsection*{Acknowledgments}
We are very 
grateful to James Davies for various comments and the proof of \cref{cor M=1 intro}. We thank Dimitrios Ntalampekos and Matthew Romney for interesting discussions on their work.

\section{Notation and preliminaries}

\subsection{Quasi-isometries and nets} \label{subsec QI}
Let $(X,d_X)$ and $(Y,d_Y)$ be metric spaces. For $M \in \R_{\geq 1}$ and $A \in \R_{\geq 0}$, an \emph{$(M, A)$-quasi-isometry} from $(X,d_X)$ to $(Y,d_Y)$ is a map~$\varphi \colon X \rightarrow Y$ such that
\begin{enumerate}
    \item \label{quasiisom:1} for every $x,y \in X$ we have
    \[\frac 1M \cdot d_X(x,y) - A \leq d_Y(\varphi(g),\varphi(h)) \leq M\cdot d_X(x,y)+A,\]
    \item \label{quasiisom:2} for every $y \in Y$ there is $x \in X$ such that $d_Y(y,\varphi(x)) \leq A$.
\end{enumerate}
If there is such a map, $(X,d_X)$ and $(Y,d_Y)$ are \emph{quasi-isometric}. It is well-known, and not hard to see, that this relation is symmetric and transitive.

\medskip

Given $x\in X$ and $r>0$, let $B(x;r)$ denote the open ball of radius $r$ around $x$. Call a set $Y\subset \Sigma$ \emph{locally finite}, if $Y \cap B$ is finite for each ball $B\subset \Sigma$ with finite radius.

A subset $Y\subset X$ is $\theta$-separated if $d_X(x,x')>\theta$ for every $x\neq x'\in X$. It is $\theta$-dense if for every $y\in X$ there is $x\in X$ with $d_X(x,y)\leq \theta$. It is a \emph{net} if it is $\theta$-separated and $\theta'$-dense for some $\theta' \geq \theta>0$. Easily, the inclusion $(Y,d_X)\hookrightarrow (X,d_X)$ is a quasi-isometry if and only if $Y$ is $\theta$-dense for some $\theta>0$.
We say that a net $Y$ of $X$ is \emph{uniform}, if for all $R\in \R$ there is a $n(R)\in \mathbb{N}$ such that for all $x \in X$, the cardinality of $Y \cap B(x;R)$ is at most $n(R)$.

\subsection{Curves}
A \emph{curve} in a metric space $(X,d_X)$ is a continuous map $\gamma\colon I\to X$ where $I\subseteq\mathbb R$ is some interval. If $I=[a,b]$, we call $\gamma(a)$ and $\gamma(b)$  the \emph{endpoints} of $\gamma$.
A curve is \emph{simple} if it is injective, and it is a \emph{simple closed curve} if its endpoints coincide but it is injective otherwise.

A curve $\gamma$ is \emph{rectifiable} if its length $\abs{\gamma}$ is finite. By a \emph{geodesic} we mean a curve $\gamma\colon [a,b]\to X$ that realizes the distance between its endpoints, i.e.\ $\abs{\gamma}=d_X(\gamma(a),\gamma(b))$.
We will often omit the parametrization and identify a curve with its image $\gamma(I)\subseteq X$. If we do refer to a parametrization for a rectifiable curve, we will generally mean the arc-length parametrization.

\subsection{Riemannian surfaces}
A \emph{surface (with boundary and corners)} is a topological surface $X$ (with boundary $\partial X$) together with a smooth atlas so that every point has a neighbourhood that is diffeomorphic to an open set in $[0,\infty)\times[0,\infty)\subset \mathbb R^2$.
A curve $\gamma\colon I\to X$ is \emph{piecewise smooth} if it is smooth except perhaps at a \emph{discrete} set of points in $I$, where a set is discrete if it has no accumulation points. In particular, every connected component of $\partial X$ is (the image of) a piecewise smooth curve in $X$.

\begin{convention}
  Throughout the paper $(\Sigma, d_\Sigma)$ stands for a complete connected Riemannian surface, possibly with boundary and corners.
\end{convention}

We will sometimes need to work with contours of subsets of $\Sigma$, and we will need them to consist of unions of simple curves. This can be problematic in general, but one can always find arbitrarily small neighbourhoods that do have this property. This fact is well-known but we did not find an appropriate reference and so we provide the following proof.

\begin{lemma}\label{lem:regular neighbourhood}
  For every subset $S\subset \Sigma$ and every $\epsilon>0$ there is a subsurface $\Sigma'\subseteq \Sigma$ that contains $S$ and is contained in its $\epsilon$-neighbourhood.
\end{lemma}
\begin{proof}[Sketch of proof]
  For every $s\in S$ there is a small neighbourhood $B_s\subset B(s;\epsilon)$ that is diffeomorphic to the radius-one ball $B(0;1)$ in $\mathbb R^2$, or $[0,\infty)\times\mathbb R$, or $[0,\infty)\times[0,\infty)$. For $t\leq1$, let $B^t_s\subseteq B_s$ be the subset corresponding to $B(0;t)$ via this diffeomorphism. Since $\Sigma$ is locally compact, there is a locally finite set $\{s_n\mid n\in \mathbb N\}\subseteq S$  such that $S\subseteq \bigcup_{\mathbb N}B_{s_n}^{1/2}$. 
  We may then exploit the local finiteness of $\{s_n\mid n\in \mathbb N\}$ to choose radii $1/2 <t_n< \epsilon$ so that the closures of the neighbourhoods $B_{s_n}^{t_n}$ may only intersect in a generic way. That is, the boundary curves of the $B_{s_n}^{t_n}$'s are never tangent to one another. The union $\bigcup_{n\in\mathbb N}B_{s_n}^{t_n}$ is then a surface as required.
\end{proof}

\subsection{Graphs and triangulations}\label{ssec: grapsh and triangulations}
A \emph{graph} is a pair $G=(V,E)$ where $V$ is a set, called the \emph{vertices}, and $E$ is a multiset of non-empty subsets of $V$ with at most two elements, called the \emph{edges}. In other words, our graphs may have \emph{loops} (edges with a single end-vertex) and several \emph{parallel} edges with the same pair of end-vertices. We can consider $G$ as a topological space by turning it into an 1-complex, i.e.\ by replacing each edge by a homeomorph of a unit interval with the same end-points.  

A \emph{metric graph} is a connected graph $ G$ together with an assignment of positive lengths to each edge. 
Such a graph is a metric space when equipped with the natural path metric.
A metric graph $G$ is \emph{locally finite} if it is proper as a metric space; equivalently, if every ball of finite radius meets finitely many edges.

An embedding of a graph $ G$ in a surface $\Sigma$ is a topological embedding of the corresponding $1$-complex. In the presence of an embedding, we will sometime abuse notation and treat $ G$ as a subset of $\Sigma$, and say that $ G\subset \Sigma$ is an an \emph{embedded graph}.
A \emph{face} of an embedded graph $G$ is a connected component of $\Sigma\smallsetminus G$. The \emph{intrinsic metric} of a face $F$ is the length metric $d_\ell^F$ inherited from $\Sigma$ via \eqref{dell}, and the \emph{intrinsic diameter} is the diameter with respect to the intrinsic metric. Note that the intrinsic metric may be arbitrarily larger than the restriction of the metric of $\Sigma$ to $F$, since there may be shortcuts via $\Sigma - F$.

\begin{remark}\label{rmk:metrizing embedded graphs}
  If $G\subset \Sigma$ is an embedded graph such that every edge $e\in E(G)$ is a rectifiable curve, then  $G$ can be made into a metric graph by assigning to every $e\in E(G)$ its length $\abs{e}$ in $(\Sigma,d_\Sigma)$. The arc-length parametrizations of the edges then define a $1$-Lipschitz embedding $G\hookrightarrow\Sigma$. In particular, the difficult part in finding embeddings that are quasi-isometric is the uniform lower bound on the metric distortion \eqref{quasiisom:1}.
\end{remark}

\begin{remark}
  Let $ G$ be a metric graph and $\sigma\colon  G\to\Sigma$ an embedding that is also Lipschitz. If the family of edges $\{\sigma(e)\mid e\in E( G)\}$ is \emph{locally finite} (\emph{i.e.}\ for every ball of finite radius $B\subseteq \Sigma$ there are only finitely many $e\in E( G)$ with $\sigma(e)\cap B\neq \emptyset$), then $ G$ is locally finite as a metric graph. If $\sigma$ is also a quasi-isometric embedding, then the converse holds as well.
\end{remark}

A \emph{cell decomposition} of $\Sigma$ is a locally finite embedded graph $G\subset \Sigma$ such that $\partial \Sigma\subseteq G$ and every face $F$ is homeomorphic to the unit open ball $B_1\subset\mathbb R^2$ via a homeomorphism $B_1\to\Sigma$ that extends to a continuous map of the  disc $\mathbb D^2=\overline{B_1}$ mapping $\partial \mathbb D^2$ into $G$ (the image of $\partial \mathbb D^2$ is $\partial F$).
This mapping defines a \emph{boundary path} in $G$, and the \emph{boundary size} of $F$ is the number of edges in its boundary path, counted with multiplicity. We say that $F$ is a \emph{$n$-gon} if it has boundary size $n$. A $n$-gon $F$ is \emph{non-degenerate} if the boundary path is a simple closed curve (i.e.\ it consists of $n$ distinct vertices and edges), otherwise it is \emph{degenerate}. A cell decomposition is a \emph{3-gonal decomposition} if every face is a $3$-gon. By a \emph{triangulation} we mean a $3$-gonal decomposition where every face is non-degenerate.

\begin{remark}
  3-gonal decomposition are called \emph{pseudo-triangulations} in \cite{Maillot}. We preferred to use the former because pseudo-triangulations of surfaces have a different meaning in computational geometry. In Hatcher's notation \cite{hatcher}, choosing a $3$-gonal decomposition would be equivalent to realizing $\Sigma$ as a (unordered) $\Delta$-complex. Other authors would call this a ``generalized triangulation'', or even a  ``triangulation''.
\end{remark}

\begin{remark}\label{rmk:barycentric subdivision of 3-gonal}
  If $G\subset \Sigma$ is a 3-gonal decomposition, we can realize the barycentric subdivision $G'$ as an embedded graph $G\subset G'\subset \Sigma$, so that $G'$ is a (simplicial) triangulation of $\Sigma$. Moreover, if $G$ is a metric graph such that the embedding $G\hookrightarrow \Sigma$ is a quasi-isometry, then we can achieve that the same holds for $G'\hookrightarrow \Sigma$.
\end{remark}

The following simple lemma converts any cell decomposition $G\subset\Sigma$ where every face has boundary size at most three into one where every face has boundary size exactly three. This is going to be very convenient in the sequel. However, when working with surfaces with boundary, we will need to assume that the intrinsic metric of the boundary is well-behaved with respect to the metric of the surface. That is, we need that $(\Sigma,d_\Sigma)$ satisfies the following:

\labtequ{comparable metric}{$d_\ell^{\partial\Sigma}$ and $d_\Sigma|_{\partial\Sigma}$ are \emph{locally Lipschitz-equivalent} metrics. That is, for every $R>0$ there is some $M_R\geq 1$ such that $d_\ell^{\partial\Sigma}(x,y)\leq M_R d_\Sigma(x,y)$ for every $x,y\in \partial\Sigma$ with $d_\Sigma(x,y)\leq R$ that belong to the same connected component of $\partial\Sigma$. 
}
(This condition is obviously vacuous if $\partial\Sigma =\emptyset$.)

\begin{lem}\label{lem: removing 2-gons}
  Given a cell decomposition $G\subset\Sigma$ where every face has boundary of size at most three and $\abs{V(G)}\geq 3$, there is a subgraph $H\subset G$ that defines a $3$-gonal decomposition of $\Sigma$. Moreover, if $\Sigma$ satisfies \eqref{comparable metric}, the following metric statements hold:
  \begin{enumerate}
    \item \label{ms i} if $G$ is a metric graph with edges of bounded length and $G\hookrightarrow \Sigma$ is a quasi-isometric embedding, then $H\hookrightarrow \Sigma$ is a quasi-isometric embedding as well;
    \item \label{ms ii} if each face of $G$ has intrinsic diameter at most $R$, then each face of  $H$ has intrinsic diameter at most $3R$.
  \end{enumerate}
\end{lem}
\begin{proof}  
  We proceed by removing parallel edges and loops bounding faces until no more are left. To do so, we enumerate $V(G)$ as $v_1, v_2, \ldots$, and perform the following procedure for $i=1,2,\ldots$, inductively assuming that after step $i-1$ each face incident with $v_1,\ldots, v_{i-1}$ has boundary of size at least three.

  Note that an edge $e$ is contained in only one face with multiplicity $1$ if and only if $e\subseteq \partial\Sigma$. Otherwise it belongs with multiplicity $1$ to two distinct faces or with multiplicity $2$ to one single face.
  
  We would like to remove all the \emph{bad} faces with boundary of size 0, 1 or 2; these can come in the following six types:
  \begin{enumerate}
    \item \label{0-gon} 0-gons ($\partial F$ is one vertex);
    \item \label{1-gon} 1-gons;
    \item \label{2-single-loop} 2-gons traversing a single loop twice;
    \item \label{2-loop} 2-gons traversing two distinct loops on the same vertex;
    \item \label{2-single-edge} 2-gons traversing a single edge in both directions;
    \item \label{2-gon} 2-gons on two distinct, parallel, edges.
  \end{enumerate}
  We start by observing that we will never meet faces of type \eqref{0-gon}, \eqref{2-single-loop} or \eqref{2-single-edge}, because each of those cases implies that $\overline F=\Sigma$, and hence $\abs{V(G)}\leq 2$ (cases \eqref{0-gon} and \eqref{2-single-edge} only happen if $\Sigma = \mathbb S^2$, while \eqref{2-single-loop} implies that $\Sigma = \mathbb PR^2$).
  
  We may now describe the inductive step. Suppose there is a bad face $F$ incident with $v_i$. It must be of type \eqref{1-gon}, \eqref{2-loop}, or \eqref{2-gon}.
  If $F$ is of type \eqref{1-gon}, we simply remove the loop $e$ contained in $\partial F$. Since $e$ is only traversed once in $\partial F$, by removing it we are merging $F$ with another face $F'$ having $e$ on its boundary ($e$ cannot be a boundary edge, as this would imply that $\Sigma=\overline F$ is a disk and $\abs{V(G)}=1$). This results in a face $F''$ that is still homeomorphic to an open ball, and has smaller boundary size than $F'$. This $F''$ is not a $3$-gon yet, but since it is still incident with $v_i$ it will be dealt with in a subsequent step of this process.

  If $F$ is of type \eqref{2-loop} or \eqref{2-gon}, let $e,e'$ be the two parallel edges in $\partial F$. As before, it cannot be that both of them are boundary edges, because this would imply $\abs{V(G)}\leq 2$. We are going to remove one of the non-boundary edges, say $e'$. If both $e$ and $e'$ are internal edges, we can assume that $e'$ is the longest one. As before, this merges $F$ with another face $F'$ having $e'$ on its boundary. Note that the resulting face is an open ball that has the same boundary size as $F'$.

  We repeat this procedure until there are no more bad faces incident with $v_i$. Note that eventually we must stop, because $G$ is locally finite and hence $v_i$ is incident with only finitely many edges, and hence faces.
  At that point all faces incident with $v_i$ (including those incident with any $v_j$, with $j<i$) are 3-gons.
  As the process goes through all $i$, it converges to a subgraph $H\subseteq G$ which is a $3$-gonal decomposition of $\Sigma$.

  Our two metric statements follow easily: for \eqref{ms i} we claim that there exists $M\geq 1$ such that $d_G(v_i,v_j)\leq d_{H}(v_i,v_j)\leq Md_G(v_i,v_j)$ for any choice of vertices $v_i,v_j$. Indeed, the leftmost inequality is obvious . For the rightmost inequality it suffices to observe that to obtain $H$ from $G$ we have:
  \begin{enumerate}
    \item removed loops;
    \item removed non-loop edges $e'$ leaving there a shorter parallel edge $e$;
    \item removed non-loop edges $e'$ connecting two vertices in a connected component of $\partial \Sigma$.
  \end{enumerate}
  The first two operations do not increase the path-distance between $v_i$ and $v_j$. The third operation may increase it, but by a bounded amount, because $\partial\Sigma$ is contained in $H$ and \eqref{comparable metric} ensures that the $d_\ell^{\partial\Sigma}$-distance between the endpoints of $e'$ is bounded by $M_R\abs{e'}$, where $R$ is an upper bound on the length of the edges in $G$.

  For \eqref{ms ii}, observe that a face $F$ of $H$ with vertices $v_1,v_2,v_3$ is obtained by glueing (finitely many) faces $F_i$ of $G$ whose vertex-sets are contained in $\{v_1,v_2,v_3\}$. Since all the $F_i$ have intrinsic diameter at most $R$, we deduce that $d(v_i,v_j)\leq R$ for every $i,j=1,2,3$ and that $\{v_1,v_2,v_3\}$ is $R$-dense in $F$. Our claim now follows from the triangle inequality.
\end{proof}

\subsection{Resolving intersections of curves}
We will later need to control how curves in a surface intersect. The following lemma allows as to push certain intersecting, but non-crossing, curves away from each other by taking a small perturbation in their neighbourhood as depicted in \cref{fig:tree intersection}. 

\begin{lemma}\label{lem: resolve intersections}
  Let $\Sigma$ be a complete Riemannian surface, possibly with boundary and corners, and let $T\subset \Sigma$ be an embedded finite rooted tree consisting of piecewise smooth curves of finite length. Let $o$ be the root and $\{p_i\mid i=1,\ldots n\}$ be the leaves. Then for every $\epsilon>0$ we may find piecewise smooth curves $\gamma_i$ joining $p_i$ to $o$ such that:
  \begin{enumerate}
    \item distinct $\gamma_i$ only meet at $o$;
    \item $\bigcup_i\gamma_i$ is within Hausdorff distance $\epsilon$ from $T$;
    \item $\abs{\abs{\gamma_i} - d_T(x_i,o)}\leq \epsilon$;
    \item \label{item:boundary condition} the curves $\gamma_i$ can only meet $\partial \Sigma$ at $o$.
  \end{enumerate}
  Moreover, if $T \cap \partial \Sigma$ consists of leaves and (possibly) the root, we may also replace \eqref{item:boundary condition} by 
  \begin{enumerate}
    \item[(4')] $(\bigcup_i\gamma_i)\cap\partial\Sigma = T\cap \partial \Sigma$.
  \end{enumerate}
\end{lemma}

\medskip
The proof of this is one of the few occasions in the paper where the Riemannian structure is helpful.

\begin{proof}[Proof of \cref{lem: resolve intersections}]
  Orient the edges of $T$ so that they point away from $o$. Subdividing edges in $T$ if necessary, we may assume that every edge $e\subseteq T$ is smooth. Moreover, we may assume that $e$ only meets $\partial \Sigma$ on one side, so that we can construct a nearby path completely contained in the interior of $\Sigma$ by ``pushing it'' in the other direction.
  Formally, this can be done by appropriately choosing a smooth variation of $e$.  
  Say that $e$ is given a constant speed parametrization $e\colon [0,1]\to \Sigma$, with $e(0)$ closer to the root $o\in T$. We consider a unit vector field $V(t)$ along $e$ that is perpendicular to $\dot e(t)$ and points towards the interior of $\Sigma$.
  Fix a very small $\epsilon'>0$ (to be determined later).
  For every $t\in [0,1-\epsilon']$, if $s$ is small enough, $\exp(stV(t))$ is a point in $\Sigma$.\footnote{
    We only define the variation on $(0,1-\epsilon']$ to avoid difficulties if $e(1)$ is a sharp corner in $\partial \Sigma$ and to ensure that different edges give rise to disjoint variations.
  }
  It follows that if we fix $0<\delta\leq \epsilon'$ small enough, the mapping $(s,t)\mapsto \exp(stV(t))$ defines a smooth variation $E\colon [0,1-\epsilon']\times[0,\delta)\to \Sigma$ (\cite{doCarmo}*{Chapter 5-4}) such that if we let $e_s(t)\coloneqq E(t,s)$ then:
  \begin{itemize}
    \item $E$ is a smooth embedding on $(0,1-\epsilon']\times[0,\delta)$;
    \item  $e(t)=e_0(t)$ for every $t\in[0,1-\epsilon']$;
    \item  $E(0,s)=e(0)$ for every $s\in [0,\delta)$;
    \item $e_s$ is at Hausdorff distance at most $2\epsilon'$ from $e$;
    \item $\abs{\abs{e_s} - \abs{e}}<2\epsilon'$.
  \end{itemize}

  \begin{figure}
    \begin{tikzpicture}[y=1cm, x=1cm, yscale=\globalscale,xscale=\globalscale, every node/.append style={scale=1}, inner sep=0pt, outer sep=0pt]

      \path[draw=black!50,line cap=butt,line join=miter,miter limit=4.0] (4.1817, 28.6422) -- (4.0585, 28.626).. controls (4.0585, 28.626) and (4.1421, 28.1658) .. (4.2985, 28.0258).. controls (4.4423, 27.804) and (5.0869, 27.6316) .. (5.0869, 27.6316) -- (5.1243, 27.7845)(4.045, 26.649) -- (3.8765, 26.5963).. controls (3.8765, 26.5963) and (3.6935, 27.0927) .. (3.7447, 27.489).. controls (4.0497, 27.5736) and (4.3162, 27.8546) .. (4.3162, 27.8546) -- (4.2077, 27.947)(2.6611, 26.7781) -- (2.519, 26.8791).. controls (2.6199, 26.999) and (2.79, 27.1599) .. (3.0365, 27.2636).. controls (3.3721, 27.2212) and (3.6526, 27.3102) .. (3.6526, 27.3102) -- (3.6018, 27.4523)(2.3712, 28.2121).. controls (2.7034, 28.1858) and (3.0943, 28.3557) .. (3.0943, 28.3557) -- (3.0207, 28.5086)(3.0365, 27.2636).. controls (2.936, 27.7003) and (2.5853, 28.1631) .. (2.5853, 28.1631) -- (2.4617, 28.0773);

      \path[draw=black, densely dotted,miter limit=4.0] (4.5549, 28.0239)arc(0.0:90.0:0.2526 and -0.2526)arc(90.0:180.0:0.2526 and -0.2526)arc(180.0:270.0:0.2526 and -0.2526)arc(270.0:360.0:0.2526 and -0.2526) -- cycle(4.0047, 27.4801)arc(0.0:90.0:0.2526 and -0.2526)arc(90.0:180.0:0.2526 and -0.2526)arc(180.0:270.0:0.2526 and -0.2526)arc(270.0:360.0:0.2526 and -0.2526) -- cycle(3.2882, 27.2562)arc(0.0:90.0:0.2526 and -0.2526)arc(90.0:180.0:0.2526 and -0.2526)arc(180.0:270.0:0.2526 and -0.2526)arc(270.0:360.0:0.2526 and -0.2526) -- cycle(2.6238, 28.2121)arc(0.0:90.0:0.2526 and -0.2526)arc(90.0:180.0:0.2526 and -0.2526)arc(180.0:270.0:0.2526 and -0.2526)arc(-90.0:0.0:0.2526 and -0.2526) -- cycle;

      \path[draw=blue,line cap=butt,line join=miter,miter limit=4.0] (3.9526, 26.6236).. controls (3.8923, 26.805) and (3.8319, 26.9861) .. (3.7883, 27.231).. controls (3.7686, 27.314) and (3.6595, 27.3145) .. (3.5731, 27.3036).. controls (3.396, 27.2771) and (3.216, 27.2723) .. (3.0365, 27.2636)(5.1037, 27.7058).. controls (4.9092, 27.7515) and (4.7148, 27.8236) .. (4.5204, 27.8974).. controls (4.4051, 27.9259) and (4.3554, 27.9204) .. (4.1963, 27.7933).. controls (4.1594, 27.76) and (4.0456, 27.6723) .. (3.9668, 27.6135).. controls (3.9098, 27.5692) and (3.8442, 27.4669) .. (3.7948, 27.4332).. controls (3.723, 27.3843) and (3.6631, 27.3702) .. (3.5406, 27.3391).. controls (3.3799, 27.3026) and (3.2065, 27.2858) .. (3.0365, 27.2636)(4.1203, 28.6377).. controls (4.137, 28.52) and (4.1681, 28.3923) .. (4.2061, 28.2597).. controls (4.247, 28.0986) and (4.2684, 27.9334) .. (4.1562, 27.8183).. controls (4.083, 27.7484) and (4.0084, 27.6856) .. (3.9494, 27.6379).. controls (3.8853, 27.5808) and (3.8261, 27.5031) .. (3.7741, 27.4636).. controls (3.7188, 27.4216) and (3.643, 27.4096) .. (3.5233, 27.3728).. controls (3.3651, 27.3256) and (3.1997, 27.2975) .. (3.0365, 27.2636)(2.584, 26.8275).. controls (2.7138, 26.9792) and (2.8699, 27.123) .. (3.0365, 27.2636)(3.0568, 28.425).. controls (2.9286, 28.37) and (2.8031, 28.3145) .. (2.6213, 28.2707).. controls (2.5236, 28.2449) and (2.5308, 28.1168) .. (2.572, 28.0629).. controls (2.76, 27.8085) and (2.8919, 27.5338) .. (3.0365, 27.2636);

      \path[draw=black,line cap=butt,line join=miter] (4.2985, 28.0258) -- (5.2671, 27.7365) (3.7447, 27.489) -- (4.0812, 26.5365)(3.0365, 27.2636) -- node[pos=0.5,yshift=15]{$T$}(3.7447, 27.489) -- (4.2985, 28.0258) -- (4.1665, 28.7749)(3.1264, 28.5554).. controls (2.5941, 28.2988) and (2.3712, 28.2121) .. (2.3712, 28.2121) -- (3.0365, 27.2636)node(root)[bnode]{} -- (2.5853, 26.6797);

      \draw (root) node[below right, yshift = -4]{$o$};

      \path[draw=black,fill=white,thick,miter limit=4.0] (1.9831, 28.4025) node[yshift=24,xshift=8]{$\partial \Sigma$} ellipse (0.4787cm and 0.3252cm);

      \path[draw=black,line cap=butt,line join=miter,thick] (3.713, 28.9769) node[yshift=4,xshift=14]{$\partial \Sigma$}.. controls (3.8015, 28.8831) and (3.9604, 28.8194) .. (4.1692, 28.7656).. controls (4.5638, 28.6639) and (4.9698, 28.7161) .. (4.9698, 28.7161).. controls (5.0344, 28.37) and (5.2394, 28.1026) .. (5.5217, 27.8787).. controls (5.4265, 27.856) and (5.3022, 27.7914) .. (5.2582, 27.7339).. controls (5.0038, 27.4016) and (5.2013, 26.8968) .. (5.3706, 26.7658);
    \end{tikzpicture}
    \caption{Transforming a rooted tree $T$ into a collection of disjoint paths meeting at the root (blue lines in the picture).}\label{fig:tree intersection}
  \end{figure}
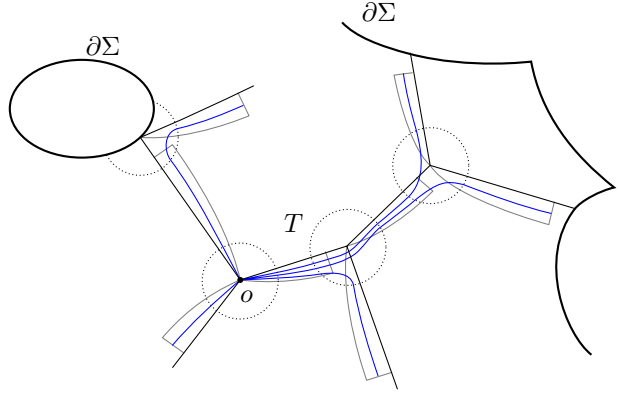

  Note that, as $s>0$ varies, the smooth curves $e_s$ only meet at $e(0)$ and are disjoint from $\partial\Sigma$ except perhaps at $e(0)$.
  Moreover, taking $\delta$ small enough we may also arrange that the curves $e_s$ and $e'_{s'}$ are disjoint as the edges $e$, $e'$ vary.

  Pick $0<s_1<\cdots < s_{n_e}$, where $n_e$ is the number of paths in $T$ going from the leaves to the root $o$ that cross $e$. As $e\in E(T)$ and $i=0,\ldots, n_e$ vary, we may now arbitrarily join up all the curves $e_{s_i}$ with appropriately chosen piecewise smooth paths contained in $2\epsilon'$-balls around the vertices of $T$ (see \cref{fig:tree intersection}). The difference in length between these new paths and the paths in $T$ is bounded by $6\abs{V(T)}\epsilon'$.
  
  The ``moreover'' part of the statement is proven by letting $s_1=0$ for every edge $e$ and judiciously joining the curves at their endpoints.
\end{proof}

\begin{remark}
  \cref{lem: resolve intersections} is the key occasion where we need the assumption that $\Sigma$ be a Riemannian surface: all the other arguments hold in much greater generality. The central property that we need for it is that every curve can be perturbed without increasing its length much. An easy example of a surface equipped with a proper geodesic metric for which this property fails was communicated to us by Dimitrios Ntalampekos. This  is constructed by equipping the plane $\mathbb R^2$ with the (discontinuous) Riemannian metric tensor obtained multiplying the Euclidean tensor by $1/2$ on the segment $[0,1]\times \{0\}$.  Then the geodesic joining $(0,0)$ to $(1,0)$ has length $1/2$, but any perturbation thereof has length at least $1$.
\end{remark}

\subsection{Arcs in minimal positions and $\epsilon$-geodesics}\label{ssec: epsilon geodesics}
Suppose that $\Sigma$ is homeomorphic to the disc $\mathbb D^2$. A curve $\gamma\colon [0,1]\to\Sigma$ is an \emph{arc} if $\{0,1\}=\gamma^{-1}(\partial \Sigma)$. A \emph{simple arc} is an arc that is a simple curve (it may be a simple closed curve meeting $\partial \Sigma$ at its endpoint).
We say that two simple arcs $\alpha$, $\beta$  are in \emph{minimal position} if $\alpha\cap\beta\smallsetminus\partial \Sigma$ consists of one point if the endpoints of $\alpha$ and $\beta$ separate each other in $\partial \Sigma\cong \mathbb S^1$, and empty otherwise (this includes the case where $\alpha$ and $\beta$ share an endpoint).

A curve $\gamma$ connecting two points $x,y\in \Sigma$ is an \emph{$\epsilon$-geodesic} if it is piecewise smooth and $\abs{\gamma} < d(x,y)+\epsilon$.

\begin{prop}\label{cor: epsilon geodesics in minimal position}
  Let $\alpha_1,\ldots, \alpha_n$ be arcs in a surface $\Sigma$ homeomorphic to $\mathbb D^2$. Then there exist piecewise smooth $\epsilon$-geodesics $\beta_i$ which are simple arcs with the same endpoints as $\alpha_i$ and are pairwise in minimal position.
\end{prop}

For the proof of this we will need

\begin{lem}\label{lem: resolving two arcs}
  Let $\Sigma$ be homeomorphic to $\mathbb D^2$ and $\alpha$, $\beta$ be piecewise smooth simple arcs that are $\epsilon_{\alpha}$ resp.\  $\epsilon_\beta$-geodesics. Let $\gamma_1,\ldots, \gamma_n$ be simple arcs that are in minimal position with respect to both $\alpha$ and $\beta$.
  Then for every $\delta>0$ there is a piecewise smooth simple arc $\beta'$ with the same endpoints as $\beta$ that is an $(\epsilon_{\alpha}+\epsilon_\beta+\delta)$-geodesic and is in minimal position with $\alpha$ and all the $\gamma_i$.
\end{lem}
\begin{proof}
  If $\alpha\cap\beta\smallsetminus\partial \Sigma$ is empty there is nothing to do. Otherwise, let $\beta''$ be the arc obtained by following $\beta$ until it first meets $\alpha$, then following $\alpha$ until its last intersection with $\beta$, and then resuming to follow $\beta$ until its other end.
  This is an $(\epsilon_\alpha+\epsilon_\beta)$-geodesic that intersects $\alpha$ in one subsegment $\xi\subseteq \alpha$.

  We claim that $\beta''$ is in minimal position with each $\gamma_i$. Let $\beta''=\beta_-\cup\xi\cup\beta_+$, and $\beta=\beta_-\cup\beta_0\cup\beta_+$. If $\gamma_i\cap\xi=\emptyset$, then $\beta''\cap \gamma_i\subseteq\beta\cap \gamma_i$ and there is nothing to show. Otherwise, $\gamma_i\cap\xi=\gamma_i\cap \alpha$ consists of exactly one point and the endpoints of alpha separate those of $\gamma_i$. But then $\gamma_i$ must intersect $\beta_0$, because the path following $\alpha$ until it reaches $\beta_0$, then $\beta_0$, and then $\alpha$ again separates the endpoints of $\gamma_i$. Since $\beta$ is in minimal position with $\gamma_i$, the same is true for $\beta''$.

  We can now apply the technique of \cref{lem: resolve intersections} to ``push $\beta''$ away from $\alpha$'': we thus find an arbitrarily small variation $\beta'$ of $\beta''$ that is in minimal position with $\alpha$. This can be done so that $\abs{\beta'\cap\gamma_i}=\abs{\beta''\cap\gamma_i}$ for every $i=1,\ldots, n$.
\end{proof}

\begin{proof}[Proof of \cref{cor: epsilon geodesics in minimal position}]
  Let $\overline \alpha_i$ be geodesics in $\Sigma$ with the same endpoints of $\alpha_i$.
  To begin with, we apply \cref{lem: resolve intersections} to each $\overline \alpha_i$  separately find arbitrarily small variations thereof that only meet $\partial P$ at their endpoints. These are simple arcs that are $\delta$-geodesic for arbitrarily small $\delta>0$.
  
  We may then iteratively fix $i=1,\ldots,n-1$ and apply \cref{lem: resolving two arcs} to each pair $\overline\alpha_i$, $\overline\alpha_j$ with $j>i$ to put all the $\overline\alpha_j$  in minimal position with $\overline\alpha_i$, while maintaining it in minimal position with all the $\overline\alpha_k$ with $k<i$. Doing this for every $i$ concludes the proof.
\end{proof}

\section{Metric triangulation results}
The main technical result of this section shows that sufficiently fine embedded graphs in a Riemannian surface can be completed to quasi-isometric 3-gonal decompositions. The precise statement requires a few technical conditions, which we introduce now.

Let $\Sigma$ be a complete Riemannian surface, possibly with boundary and corners. We fix once and for all two constants:
\labtequ{parameters}{Fix (small) parameters $\theta$ and $\Theta$ with $0<\theta<\Theta/2$.}
\begin{remark}
  We could \emph{e.g.}\ fix $\theta =\Theta/3$, but we prefer to keep the parameters independent as they represent two different quantities: the density of vertices vs.\ the maximal length of edges.
\end{remark}
We also fix a ``selection of vertices'':
\labtequ{X net}{let $X\subset \Sigma$ be a locally finite subset that is $\theta$-dense and such that $X\cap \partial\Sigma$ is $\theta$-dense on each connected component of $\partial\Sigma$ with respect to the intrinsic metric (the last condition is vacuous if $\partial\Sigma=\emptyset$).}

Let $\overline G\subset \Sigma$ be a locally finite embedded graph such that $\partial \Sigma \subset \overline G$ and $X\subseteq V(\overline G)$.
We further assume that
\labtequ{nice metric graph}{when equipped with its arc-length metric, $\overline G$ has edge-lengths at most $\Theta$ and $X$ is $\Theta/2$ dense in it.}

One key extra property that we require on $\overline G$ is that it is sufficiently fine to represent homotopy classes of curves in the following sense:

\labtequ{homotopy approximation}{
  There exist constants $M,A$ such that if $\gamma\colon [0,1]\to \Sigma$ is a rectifiable curve with endpoints in $X$, then there exists a path $\bar\gamma$ in $\overline G\subset \Sigma$ such that
  \begin{enumerate}
    \item $\gamma$ and $\bar\gamma$ are homotopic within the $2\Theta$-neighbourhood of $\gamma$;
    \item 
      \(
      \abs{\bar\gamma}\leq M\abs{\gamma} + A.
      \)
  \end{enumerate}

}

We observe that this property already implies that $\overline G$ is quasi-isometric to $\Sigma$:

\begin{lemma}\label{lem: quasi isometry}
  The embedding $(G,d_\ell^G)\hookrightarrow (\Sigma, d_\Sigma)$ is a quasi-isometry.
\end{lemma}
\begin{proof}
  By construction, we already know that the embedding is $1$-Lipschitz. Conversely, let $x,y$ be points in $X$. Applying \eqref{homotopy approximation} to a geodesic connecting them in $\Sigma$, we deduce that
  \begin{equation}\label{eq:coarse embedding}
  (d_{\overline G}(x,y) - A)/M \leq d_\Sigma(x,y)\leq d_{\overline G}(x,y).
  \end{equation}
  The claim follows, because $X$ is coarsely dense in both $\overline G$ and $\Sigma$.
\end{proof}

A graph $\overline{ G}$ satisfying all these hypotheses can be constructed by fixing an appropriate net $X$ and ``drawing in'' paths connecting nearby points in $X$ (see \cref{ssec:construction of G} below). Once such a graph is given, we prove:

\begin{thm}\label{thm: filling triangulation}
 Given $X\subset \overline G\subset \Sigma$ as above, $\overline G$ defines a cell decomposition of $\Sigma$.
 It is possible to add edges of bounded length to $\overline G$ to obtain a locally finite cell decomposition  $G'\subset\Sigma$ such that
 \begin{enumerate}
  \item \label{G' iv} every face of $G'$ has at most three edges;
  \item \label{G' iii} the faces $F$ of $G'$ have uniformly bounded diameter with respect to their intrinsic metric;
  \item  $ G'\hookrightarrow \Sigma$ is a quasi\=/isometry. 
 \end{enumerate}
 Moreover, if $\partial \Sigma$ satisfies \eqref{comparable metric}, then we can remove edges from $G'$ to obtain a $3$-gonal decomposition $G\subset \Sigma$ satisfying the same properties.
\end{thm}

Judiciously choosing the starting data, we obtain the following.

\begin{thm} \label{thm: triangulation}
 Let $\Sigma$ be a complete Riemannian surface, possibly with boundary and corners.
 For every $\Xi>0$ there exists a locally finite metric graph $G$ with a $1$-Lipschitz embedding $ G\hookrightarrow \Sigma$ such that: 
 \begin{enumerate}
  \item \label{G iv} $G$ defines a $3$-gonal decomposition of $\Sigma$;
  \item \label{G ii} every edge of $G$ has length at most $\Xi$;
  \item \label{G iii} every face $F$ of $G$ has diameter at most $\Xi$ with respect to its intrinsic metric $d^F_\ell$.
 \end{enumerate}
 Moreover, if $\partial \Sigma$ satisfies \eqref{comparable metric}, then
 \begin{enumerate}\setcounter{enumi}{3}
  \item \label{G i}  $ G\hookrightarrow \Sigma$ is a quasi\=/isometry. 
 \end{enumerate}
\end{thm}

By \cref{rmk:barycentric subdivision of 3-gonal}, applying a barycentric subdivision to the $3$-gonal decomposition of \cref{thm: triangulation}, we obtain:

\begin{cor}
  Under the hypotheses of \cref{thm: triangulation}, we can further achieve that $G$ induces a triangulation of $\Sigma$.
\end{cor}

\begin{remark}
  It is plausible that \cref{thm: triangulation} remains true without the assumption \eqref{comparable metric}, but this would introduce substantial technical difficulties. We thus decided to work under this restriction, which is satisfied by all natural surfaces we could think of.
\end{remark}

The rest of this section is devoted to the proof of \cref{thm: filling triangulation,thm: triangulation}.

\subsection{Bounding the extrinsic diameter of the faces}
We shall now show that $\Sigma\smallsetminus \overline G$ has connected components of uniformly bounded diameter, which hence have compact closure.

\begin{prop}\label{prop:bounded diameter}
  Each face of $\overline G\subset \Sigma$ has $d_\Sigma$-diameter at most $13\Theta$.
\end{prop}

\begin{proof}
 Suppose to the contrary there is a face $F$  containing two points $p,p'$ with $d_\Sigma(p,p')=13\Theta$ and pick a smooth, simple curve $c\subset F$ joining $p$ to $p'$. We will reach a contradiction by showing that there is an edge in $\overline G$ crossing $c$.
 
 The constant $13\Theta$ is large enough that we can split $c$ into sub-segments as $c= c_-\cup c_m\cup c_+$ such that
 \begin{enumerate}
  \item\label{item:segments are apart} $d_\Sigma(c_-,c_+)\geq 2\theta $;
  \item\label{item:segments are long} both the endpoints of $c_-$ and those of $c_+$ are at $d_\Sigma$-distance greater than $(4\Theta+3\theta)$ from one another.
 \end{enumerate}
 
 Let $N_0$ be the closed $2\theta$\=/neighbourhood of $c_m$. Enlarging $N_0$ a little if necessary, we may assume that it is a (compact) subsurface of $\Sigma$ (\cref{lem:regular neighbourhood}). Let $c_0$ be the closure of the connected component of $c\smallsetminus \partial N_0$ containing $c_m$.
 
 Let $N_1$ be the closed $(2\Theta+\theta)$\=/neighbourhood of $N_0$ (\cref{fig:construction}). We again assume that it is a submanifold of $\Sigma$ and let $c_1$ be the closure of the connected component of $c\smallsetminus \partial N_1$ containing $c_0$.
 Similarly, $N_2$ is the $2\Theta$\=/neighbourhood of $N_1$ and $c_2$ is the closure of the component of $c\smallsetminus \partial N_2$ containing $c_1$.
 Note that condition \eqref{item:segments are long} implies that the curves $c_i$ are properly nested, and that for each $i=0,1,2$ the curve $c_i$ meets the boundary $\partial N_i$ in two points $p_i^-\in c_-$ and $p_i^+\in c_+$.

 \smallskip

 We now distinguish two cases depending on whether $N_1\smallsetminus c_1$ is connected.
 
 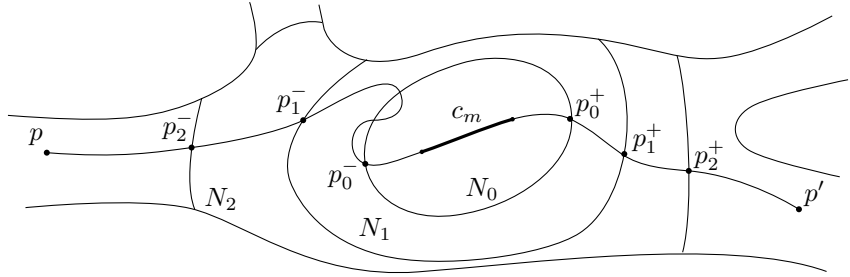
\begin{figure}
  \begin{tikzpicture}[y=1cm, x=1cm, yscale=\globalscale,xscale=\globalscale, every node/.append style={scale=1}, inner sep=0pt, outer sep=0pt]
    \path[draw=black,line cap=butt,line join=miter,] (6.8797, 24.0147).. controls (6.2294, 24.2502) and (4.8806, 24.0653) .. (4.1639, 24.0121).. controls (3.4473, 23.959) and (3.001, 24.3511) .. (2.6531, 24.4335).. controls (2.3051, 24.516) and (1.5549, 24.4394) .. (1.5549, 24.4394)(7.0227, 25.2904).. controls (7.0227, 25.2904) and (6.564, 25.2015) .. (6.4007, 25.1226).. controls (6.2374, 25.0436) and (6.2382, 24.8986) .. (6.3625, 24.8222).. controls (6.4867, 24.7457) and (6.9674, 24.6445) .. (6.9674, 24.6445)(3.5054, 25.7896).. controls (3.54, 25.6441) and (3.5263, 25.6299) .. (3.5758, 25.5596).. controls (3.6254, 25.4892) and (3.7971, 25.3281) .. (4.1953, 25.4663).. controls (4.5935, 25.6044) and (4.9299, 25.6196) .. (5.3311, 25.5646).. controls (5.7323, 25.5097) and (5.9247, 25.3965) .. (6.2131, 25.4397).. controls (6.5014, 25.4828) and (6.9172, 25.7149) .. (6.9172, 25.7149)(1.4383, 25.0532).. controls (2.0582, 25.0061) and (2.3472, 25.02) .. (2.626, 25.1235).. controls (2.9048, 25.227) and (3.081, 25.3659) .. (3.0879, 25.5116).. controls (3.0947, 25.6573) and (3.0447, 25.8062) .. (3.0447, 25.8062);

    \path[draw=black,line cap=butt,line join=miter,] (5.8864, 25.4487).. controls (5.9596, 25.1788) and (6.0013, 24.4411) .. (5.923, 24.1426)(3.0879, 25.4912).. controls (3.201, 25.6142) and (3.3686, 25.7014) .. (3.527, 25.6706)(2.6797, 24.4292)node[yshift=4,xshift=10]{$N_2$}.. controls (2.5846, 24.6663) and (2.7272, 25.1681) .. (2.7272, 25.1681);

    \path[draw=black,line cap=butt,line join=miter,] (3.8254, 25.4242).. controls (3.3837, 25.1349) and (3.2438, 24.808) .. (3.3137, 24.6155).. controls (3.3689, 24.4637) and (3.5421, 24.2187) .. (3.9141, 24.1236)node[yshift=10,xshift=-2]{$N_1$}.. controls (4.3075, 24.023) and (5.0098, 24.1192) .. (5.2405, 24.2719).. controls (5.5939, 24.5059) and (5.6691, 25.2855) .. (5.3617, 25.5621);

    \path[draw=black,line cap=butt,line join=miter,] (4.459, 25.4017).. controls (4.1524, 25.3116) and (3.9105, 25.1586) .. (3.8345, 24.9445).. controls (3.7194, 24.6203) and (3.9893, 24.3311) .. (4.4136, 24.3893)node[yshift=10,xshift=10]{$N_0$}.. controls (4.8379, 24.4475) and (5.1558, 24.7178) .. (5.1826, 24.9673).. controls (5.2289, 25.3987) and (4.7656, 25.4918) .. (4.459, 25.4017) -- cycle;

    \path[draw=black,line cap=butt,line join=miter,] (4.7924, 25.0276).. controls (4.9333, 25.0689) and (5.0907, 25.0724) .. (5.1747, 25.03)node[bnode]{}node[xshift=8,yshift=6]{$p_0^+$}.. controls (5.34, 24.9467) and (5.401, 24.8768) .. (5.5178, 24.7963)node[bnode,xshift=1]{}node[xshift=10,yshift=6]{$p_1^+$}.. controls (5.6452, 24.7085) and (5.8169, 24.6968) .. (5.9637, 24.684)node[bnode]{}node[xshift=8,yshift=6]{$p_2^+$}.. controls (6.3547, 24.6499) and (6.6962, 24.4286) .. (6.6962, 24.4286)node[bnode]{}node[xshift=6,yshift=6]{$p'$};

    \path[draw=black,line cap=butt,line join=miter, very thick,miter limit=4.0] (4.1894, 24.8104)node[sbnode]{}.. controls (4.3923, 24.8793) and (4.5928, 24.969) .. node[pos=0.5,yshift=8]{$c_m$}(4.7924, 25.0276)node[sbnode]{};

    \path[draw=black,line cap=butt,line join=miter,] (1.6966, 24.8051)node[bnode]{}node[xshift=-4,yshift=6]{$p$}.. controls (2.1925, 24.7665) and (2.4321, 24.8095) .. (2.659, 24.8391)node[bnode]{}node[xshift=-5,yshift=8]{$p_2^-$}.. controls (2.8751, 24.8672) and (3.2323, 24.9364) .. (3.401, 25.021)node[bnode]{}node[xshift=-4,yshift=8]{$p_1^-$}.. controls (3.64, 25.1408) and (3.9729, 25.3674) .. (4.0441, 25.206).. controls (4.0884, 25.1055) and (4.0307, 25.021) .. (3.8683, 25.021).. controls (3.706, 25.021) and (3.6765, 24.7804) .. (3.8137, 24.7306)node[bnode]{}node[xshift=-8,yshift=-4]{$p_0^-$}.. controls (3.9038, 24.698) and (4.0378, 24.759) .. (4.1894, 24.8104);
  \end{tikzpicture}
  \caption{Constructing nested surfaces around the middle portion of a very long curve $c$.}\label{fig:construction}
 \end{figure}

\smallskip

  \emph{Case I}.
 Assume $N_1\smallsetminus c_1$ is connected.
 Since $X=V(\overline{G})$ is $\theta$-dense, we can choose a point $x\in X\cap N_1$. We can also choose a simple closed curve $\gamma\subset N_1$ that starts at $x$ and intersects $c_1$ at exactly one point (head towards $c_1$, cross it, and return to $x$ using that $N_1\smallsetminus c_1$ is path connected).
 Let $\bar\gamma$ be the path in $\overline G$ given by \eqref{homotopy approximation}. By construction, the homotopy between $\gamma$ and $\bar\gamma$ takes place within $N_2$. Since $c_2$ is a simple arc joining two points in $\partial N_2$ and $\gamma\cap c_2=\gamma\cap c_1$ is one point, it follows that $\bar\gamma$ must also intersect $c_2$ in an odd number of points (this can be seen in several ways, for instance with a homological argument, or by taking the double of $N_2$ in order to prolong $c_2$ to a closed curve and apply standard intersection theory see e.g.\ \cite{guillemin2025differential}*{Section 2.4}). Hence $c$ crosses an edge of $\overline G$, contradiction.
 
\smallskip

 \begin{figure}
  \begin{tikzpicture}[y=1cm, x=1cm, yscale=\globalscale,xscale=\globalscale, every node/.append style={scale=1}, inner sep=0pt, outer sep=0pt]
    \path[draw=black,thick,line cap=butt,line join=miter,] (3.4932, 23.2845).. controls (2.9261, 23.2634) and (2.5171, 23.0689) .. (2.5042, 22.6309).. controls (2.4914, 22.1928) and (3.0107, 21.9226) .. (3.4241, 21.9289).. controls (3.9908, 21.9374) and (4.4958, 22.1944) .. (4.5161, 22.5752)node[xshift=12,yshift=-22]{$C\subseteq \partial N_0$}.. controls (4.5404, 23.0327) and (4.0603, 23.3056) .. (3.4932, 23.2845) -- cycle;

    \path[draw=black,dashed,miter limit=4.0] (3.4932, 23.2845)node[bnode]{}node[yshift=8]{$q_+$} circle (0.3571cm);

    \path[draw=black,dashed,miter limit=4.0] (3.3832, 21.929) circle (0.3571cm)node[bnode]{}node[xshift=6,yshift=-8]{$q_-$};

    \path[draw=black,line cap=butt,line join=miter,] (2.3534, 23.4943).. controls (1.852, 23.2528) and (1.606, 22.4767) .. (1.9338, 22.0466).. controls (2.2616, 21.6165) and (2.9272, 21.4971) .. (3.3262, 21.4954).. controls (4.2319, 21.4918) and (4.7862, 21.7668) .. (5.1009, 22.0657).. controls (5.3481, 22.3005) and (5.3546, 23.1531) .. (4.6137, 23.4986)node[yshift=6,xshift=8]{$N_1$}.. controls (3.8276, 23.8652) and (2.8547, 23.7358) .. (2.3534, 23.4943) -- cycle;

    \path[draw=black,line cap=butt,line join=miter,] (3.93, 22.6487).. controls (4.2177, 22.6731) and (4.4243, 22.6451) .. (4.5107, 22.6284)node[bnode]{}node[xshift=8,yshift=8]{$p_0^+$}.. controls (4.7204, 22.5877) and (4.9957, 22.5462) .. (5.2384, 22.5573)node[bnode]{}node[xshift=8,yshift=8]{$p_1^+$}.. controls (5.3706, 22.5634) and (5.5347, 22.566) .. (5.5347, 22.566);

    \path[draw=black,line cap=butt,line join=miter,very thick,miter limit=4.0] (3.1179, 22.6172)node[sbnode]{}.. controls (3.3737, 22.6307) and (3.681, 22.6276) .. (3.93, 22.6487)node[sbnode]{};

    \path[draw=black,line cap=butt,line join=miter,] (1.5155, 22.736).. controls (1.5992, 22.7389) and (1.7509, 22.7614) .. (1.8065, 22.7631)node[bnode]{}node[xshift=-4,yshift=8]{$p_1^-$}.. controls (2.125, 22.7726) and (2.1933, 22.7826) .. (2.5141, 22.6985)node[bnode]{}node[xshift=-4,yshift=9]{$p_0^-$}.. controls (2.6279, 22.6687) and (2.835, 22.6024) .. (3.1179, 22.6172);

    \path[draw=black,line cap=butt,line join=miter,thick] (3.1772, 21.7762)node[bnode]{}node[yshift=-6,xshift=6]{$x_-$}.. controls (3.1719, 21.9219) and (3.3348, 22.1249) .. (3.4359, 22.2169).. controls (3.537, 22.3089) and (3.6299, 22.4765) .. (3.6607, 22.5507)node[xshift=4,yshift=-8]{$\gamma$}.. controls (3.7283, 22.7139) and (3.6881, 23.0718) .. (3.7491, 23.1798)node[bnode]{}node[xshift=-8,yshift=-6]{$x_+$};
  \end{tikzpicture}
  \caption{Constructing a curve $\gamma$ with endpoints in $X$ that cut $c$ in the case that $N_1\smallsetminus c_1$ is disconnected.}\label{fig:disconnected case}
 \end{figure}
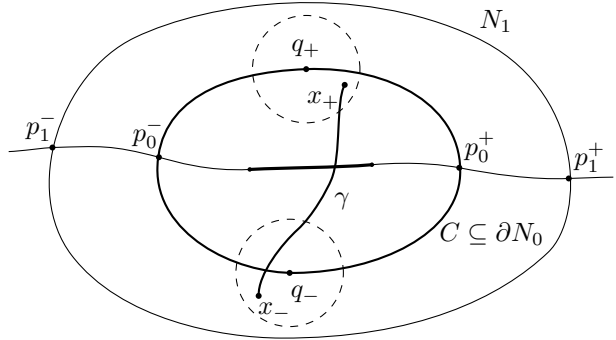
 \emph{Case II} (\cref{fig:disconnected case}).
 Assume $N_1\smallsetminus c_1$ is disconnected.
 Then $N_0\smallsetminus c_0$ is a fortiori disconnected (every point in $N_1$ is connected to $N_0$ via a path avoiding $c$). It follows that $p_0^+$ and $p_0^-$ must belong to the same component $C$ of $\partial N_0$ (if a simple arc $c'$ in a connected surface $\Sigma'$ meets a boundary component $C'$ at a single point, then $\Sigma'\smallsetminus c'$ is path connected because so is $C'\smallsetminus c'$).
 
 Note that $C$ is a circle and $p_0^+$ and $p_0^-$ cut it into two segments, each of which connects $c_-$ to $c_+$. By continuity, \eqref{item:segments are apart} implies that each of these segments contain points that are at distance greater than $\theta$ from both $c_-$ and $c_+$.
 Let $q$ and $q'$ be two such points, one for each segment.
 Since every point on $C$ is at distance greater than $\theta$ from $c_m$, it follows that $q$ and $q'$ are at distance greater than $\theta$ from $c=c_-\cup c_m\cup c_+$.
 Let $x_-$ and $x_+$ be points in $X$ closest to $q_-$ and $q_+$ respectively. Since the balls $B(q;\theta)$ and $B(q';\theta)$ are contained in $N_1\smallsetminus c$, the points $x$ and $x'$ belong to different components of $N_1\smallsetminus c_1$.

 We arbitrarily choose a curve $\gamma$ joining $x_-$ with $x_+$ without leaving the (closed) $\theta$\=/neighbourhood of $N_0$, and apply \eqref{homotopy approximation} to obtain a curve $\bar \gamma$ in $\overline G$ homotopic to it. The homotopy between $\gamma$ and $\bar\gamma$ takes place in $N_1$ by construction, and in particular $\bar\gamma$ is a curve connecting $x$ and $x'$ in $N_1$. Since $x$ and $x'$  are in two distinct components of $N_1\smallsetminus c_1$, it follows that $\bar\gamma$ intersects $c$.
\end{proof}

\subsection{Proving that faces are discs}\label{ssec: bounding topology}
We shall now show that $\overline{G}$ defines a cell decomposition of $\Sigma$ (\cref{ssec: grapsh and triangulations}).
To prove this, we start with a few preliminary observations, which will be of use both here and in the next section.

Suppose $G\subset \Sigma$ is some embedded graph with $\partial \Sigma \subseteq G$ whose edges are piecewise smooth and have finite length.
Given a face $F$ of $G$ with compact closure, we can consider it with its own intrinsic path-metric $d_\ell^F$ and denote by $\widehat{F}$ its completion. 
Observe that $\widehat{F}$ will generally differ from the closure $\overline F\subseteq\Sigma$, and it is a compact Riemannian surface with boundary and corners (taking the completion of the interior with respect to the intrinsic metric has the effect of ``opening up'' non-trivial glueings of the boundary). In particular, $\partial\widehat F\coloneqq \widehat{F}\smallsetminus F$ is homeomorphic to a disjoint union of loops.
Note that the inclusion $F\hookrightarrow\Sigma$ extends to a $1$-Lipschitz surjective map $p\colon(\widehat{F},d_F)\to (\overline F, d_\Sigma)$ such that $p^{-1}(\partial F)=\partial\widehat F$. In turn, $p$ descends to a homeomorphism when quotienting out the boundary
\[
\widehat{F}/\partial \widehat{F}\cong \overline{F}/\partial F \cong \Sigma/(\Sigma\smallsetminus F).
\]

\smallskip

Let now $F$ be a face of the embedded graph $\overline{G}$ constructed in \cref{ssec:construction of G}. To prove that $\overline{G}$ is a cell decomposition it is enough to show that $\widehat{F}$ is homeomorphic to the disc $\mathbb D^2$.
Let $\pi\colon \widehat{F} \to\widehat{F}/\partial \widehat{F}$ be the quotient map.
We observe the following.

\begin{lemma}\label{lem:paths null-homotopic in the quotient}
  If $\gamma\colon I\to \widehat F$ is a path with endpoints in $\partial \widehat F$, then the closed loop $\pi\circ \gamma\colon I \to \widehat{F}/\partial \widehat{F}$ is null-homotopic.
\end{lemma}
\begin{proof}
  Consider the path $p\circ \gamma$ in $\Sigma$ and note that if we identify $\widehat F/\partial\widehat F\cong \Sigma/(\Sigma\smallsetminus F)$ then $\pi\circ\gamma = \pi_\Sigma\circ p\circ \gamma$, where $\pi_\Sigma$ is the quotient map $\Sigma\to\Sigma/(\Sigma\smallsetminus F)$. We may prolong $p\circ\gamma$ along $\overline G$ to obtain a path $ \gamma'$ with endpoints in $X$, and we observe that $\pi\circ p\circ \gamma$ and $\pi\circ \gamma'$ are homotopic as closed loops.
  By \eqref{homotopy approximation}, $\gamma'$ is homotopic to a path $\overline\gamma'\subseteq\overline G\subseteq \Sigma\smallsetminus F$.
  But then we are done, because $\pi\circ\overline\gamma'$ is constant in $\Sigma/(\Sigma\smallsetminus F)$. 
\end{proof}

Given the the classification of compact surfaces, the following fact is an exercise in algebraic topology (which can be solved using either fundamental groups or homology computations).

\begin{fact}\label{lem:recognize disks}
  A connected compact surface $\Delta$ with $\partial \Delta\neq \emptyset$ is homeomorphic to a disk if and only if the image of every path $\gamma\colon I\to \Delta$ with endpoints in $\partial \Delta$ under the quotient map $\pi\colon \Delta\to\Delta/\partial\Delta$ is null-homotopic (\emph{i.e.}\ $\pi\circ\gamma=0\in \pi_1(\Delta/\partial\Delta)$).
\end{fact}

Combining \cref{lem:recognize disks} with \cref{lem:paths null-homotopic in the quotient}, we obtain:

\begin{cor}\label{cor: cell decomposition}
  $\overline G\subset\Sigma$ defines a cell decomposition of $\Sigma$.
\end{cor}

\subsection{Cutting into a 3-gonal decomposition}
Having just proved that $\overline{G}\subset \Sigma$ induces a cell decomposition, it is now straightforward to obtain a 3-gonal decomposition $G\subset \Sigma$ by adding extra edges if necessary. However, since \cref{prop:bounded diameter} does not bound the \emph{intrinsic} diameter, the newly added edges may be too long, thus spoiling the metric properties of the embedding. The aim of this subsection is to address this difficulty and establish \eqref{G iii} of \cref{thm: triangulation} while preserving \eqref{G ii}.
\smallskip

Let $F$ be a face of $\overline{G}$, let $\widehat{F}$ be its completion with respect to the intrinsic metric $d^F_\ell$, and $p\colon \widehat{F}\to \overline F$ as in \cref{ssec: bounding topology}.
Observe that the graph structure of $\partial F\subseteq \overline G$ lifts to a graph structure on $\partial\widehat{F}$, where vertices are preimages of vertices and edges are lifts of  edges. Moreover, every edge $e\in \partial \widehat{F}$ has the same length as its image $p(e)\in E(\overline G)$.
Note that $\widehat F$ is a \emph{piecewise smooth polygon} with edge-lengths bounded by $\Theta$. That is, it is a disc whose boundary (considered as a graph) is a cycle consisting of finitely many edges, each of which is piecewise smooth and has length bounded by $\Theta$. Moreover, since $X$ is $\theta$-dense, the polygon $\widehat F$ is also \emph{$\theta$-thin}; that is, every point in $\widehat{F}$ is within distance $\theta$ from $\partial \widehat{F}$.
To conclude the proof of \cref{thm: triangulation}, it is then enough to triangulate such a polygon by edges of controlled length.

\smallskip

Recall that a curve is an $\epsilon$-geodesic if its length realizes the distance of its endpoints up to $\epsilon$ (\cref{ssec: epsilon geodesics}).

\begin{lem}\label{lem: thin geodesic polygons}
If $e\subset\partial P$ is an edge of a $\theta$-thin polygon $P$, and $e$ is an $\epsilon$-geodesic, then every point in $e$ is within distance $2\epsilon+3\theta$ from $\partial P\smallsetminus e$.  
\end{lem}
\begin{proof}
  Given $x\in e$, let $N$ be a subsurface with $B(x;\epsilon+2\theta)\subseteq N\subseteq B(x;2\epsilon+2\theta)$, provided by \cref{lem:regular neighbourhood}. If $N$ intersects $\partial P\smallsetminus e$, there is nothing to do.
  If that is not the case, note that $\partial N$ is contained in $e\cup (P\smallsetminus B(x;\epsilon+2\theta))$ (here $\partial N$ denotes the surface boundary, not the topological boundary of the subset $N\subset P$).
  Let $e_0$ be the component of $e \cap B(x;\epsilon/2+\theta)$ containing $x$, and decompose $e$ into (non-empty) segments as $e=e_-\cup e_0\cup e_+$.
  
  Observe that no point in $\partial N\smallsetminus e$ can be within distance $\theta$ from both $e_-$ and $e_+$, because otherwise $e$ would not be an $\epsilon$-geodesic since $\abs{e_0}\geq 2\theta+\epsilon$. On the other hand the component of $\partial N$ containing $x$ is a simple closed curve $\gamma$ that contains $e_0$. It follows that $\gamma\smallsetminus e$ must have a component $\gamma'$ that joins a point in $e_-$ to a point in $e_+$.
  Since $\gamma'$ is contained in $P\smallsetminus B(x;\epsilon+2\theta)$, it contains points at distance greater than $\theta$ from $e_-$ and $e_+$ respectively, and since it is connected it must also contain a point that is at distance greater than $\theta$ from both $e_-$ and $e_+$ simultaneously.
  That point is also at distance greater than $\theta$ from $e_0\subset B(x;\epsilon/2+\theta)$, so it is not $\theta$-close to $e$, and must hence be in the $\theta$-neighbourhood of $\partial P\smallsetminus e$.
\end{proof}

\begin{prop}\label{prop: triangulate polygons}
  Let $\theta<\Theta/2$ and let $P$ be a $\theta$-thin piecewise smooth $n$-gon with edge-length bounded by $\Theta$ and $n>  3$.
  We may triangulate $P$ by adding finitely many arcs of length at most $12(3\theta + \Theta)$ and no new vertices.
\end{prop}
\begin{proof}
  Let $\kappa\coloneqq 3\theta + \Theta$.
  We arbitrarily fix a small $0<\epsilon< \Theta/4$.
  Let $A=\{\alpha_i\mid i=1,\ldots ,n\}$ be a family of curves of maximal cardinality such that
  \begin{itemize}
    \item each $\alpha_i$ is a piecewise smooth $\epsilon$-geodesic of length less than $12\kappa$ that join non-adjacent vertices in $\partial P$ and is otherwise contained in the interior of $P$;
    \item no pair of vertices in $\partial P$ is joined by more than one curve in $A$;
    \item the curves $\alpha_i$ do not intersect in the interior of $P$.
  \end{itemize}
  We claim that $A$ yields the required triangulation of $P$. 
  
  We prove the claim by contradiction, assuming that $P\smallsetminus A$ has a face $F$ that is not a triangle.
  Observe that  
  \labtequ{alpha_zero}{$\partial F$ must contain an edge $\alpha_0$ of length at least $6\kappa$,} because otherwise we may dissect $F$ further using $\epsilon$-geodesics of length less than $12\kappa$ by joining non-consecutive vertices in $\partial P$ at distance less than $12\kappa$ (\cref{cor: epsilon geodesics in minimal position}).
  Since the edges in $\partial P$ have length at most $\Theta$, we deduce that $\alpha_0$ must be one of the curves in $A$.

  The $\epsilon$-geodesic $\alpha_0$ cuts $P$ into two piecewise smooth polygons, $P'$ and $P''$, both of which are $\theta$-thin. We may assume that $F$ is contained in $P'$. Let $v_-,v_+\in\partial P'$ be the endpoints of $\alpha_0$.  
  By \cref{lem: thin geodesic polygons} applied to $P'$, every point $m\in\alpha_0$ is within distance $3\theta+\epsilon$ from $\partial P'\smallsetminus \alpha_0$. Therefore, there must be a vertex in $\partial P'$ that is at distance at most 
  \begin{equation}\label{eq:distave of w_m}
    3\theta +\Theta/2+\epsilon < \kappa
  \end{equation}
  from $m$.
  Let $w_m$ be a vertex of $P'$ that is closest to $m$. 

  \begin{figure}    
    \begin{tikzpicture}[y=1cm, x=1cm, yscale=\globalscale*1.2,xscale=\globalscale*1.2, every node/.append style={scale=1}, inner sep=0pt, outer sep=0pt]
      \path[draw=black,line cap=butt,line join=miter,] (8.2649, 27.3177)node[bnode]{}.. controls (8.2851, 27.5292) and (8.388, 27.7222) .. (8.1711, 27.9865)node[bnode]{}.. controls (8.4372, 28.0852) and (8.4728, 28.3184) .. (8.5264, 28.5411)node[bnode]{}.. controls (8.7804, 28.3985) and (9.0364, 28.2922) .. (9.3038, 28.3821)node[bnode]{}.. controls (9.5459, 28.3402) and (9.7814, 28.4058) .. (10.0089, 28.6039)node[bnode]{}.. controls (10.1601, 28.4521) and (10.312, 28.3011) .. (10.5464, 28.2594)node[bnode]{}.. controls (10.4124, 28.1221) and (10.298, 27.975) .. (10.403, 27.7189)node[bnode]{}.. controls (10.2297, 27.6097) and (10.0345, 27.5318) .. (9.9572, 27.2855)node[bnode]{}.. controls (9.5287, 27.3099) and (9.2745, 27.1695) .. (8.9975, 27.0505)node[bnode]{}.. controls (8.7897, 27.2455) and (8.5451, 27.3336) .. (8.2649, 27.3177) -- cycle;

      \fill[black!15](8.5264, 28.5411).. controls (8.6749, 28.0487) and (8.5334, 27.6604) .. node[pos=0.5,black,xshift=8]{$F$}(8.2649, 27.3177)
        .. controls (8.9068, 27.5757) and (9.6297, 27.6916) .. (10.403, 27.7189)
        .. controls (10.1624, 27.9639) and (9.9745, 28.2331) .. (10.0089, 28.6039)
        .. controls (9.5323, 27.8984) and (9.0355, 27.9808) .. (8.5264, 28.5411);

      \path[draw=black,thick,line cap=butt,line join=miter,] (8.2649, 27.3177).. controls (8.9068, 27.5757) and (9.6297, 27.6916) .. node[pos=0.3,yshift=-8]{$\alpha_0$}(10.403, 27.7189);

      \path[draw=black,thick,line cap=butt,line join=miter,] (8.5264, 28.5411).. controls (8.6749, 28.0487) and (8.5334, 27.6604) .. (8.2649, 27.3177);

      \path[draw=black,thick,line cap=butt,line join=miter,] (10.0089, 28.6039).. controls (9.5323, 27.8984) and (9.0355, 27.9808) .. (8.5264, 28.5411);

      \path[draw=black,thick,line cap=butt,line join=miter,] (10.403, 27.7189).. controls (10.1624, 27.9639) and (9.9745, 28.2331) .. (10.0089, 28.6039);

      \path[draw=black,thick,line cap=butt,line join=miter,] (8.9975, 27.0505).. controls (9.6698, 27.5962) and (9.9949, 27.5918) .. (10.403, 27.7189);

      \path[draw=black,line cap=butt,line join=miter,] (9.3265, 27.6075)node[bnode]{}node[yshift=-8]{$m$}.. controls (9.2432, 27.9339) and (9.2942, 28.3171) .. node[pos=0.3,xshift=-8]{$\beta_m$}(9.3038, 28.3821)node[yshift=8]{$w_m$};


      \draw  (10.75, 27.8)node{$\Longrightarrow$};

      \begin{scope}[xshift=80]
        
      \path[draw=black,line cap=butt,line join=miter,] (8.2649, 27.3177)node[bnode]{}.. controls (8.2851, 27.5292) and (8.388, 27.7222) .. (8.1711, 27.9865)node[bnode]{}.. controls (8.4372, 28.0852) and (8.4728, 28.3184) .. (8.5264, 28.5411)node[bnode]{}.. controls (8.7804, 28.3985) and (9.0364, 28.2922) .. (9.3038, 28.3821)node[bnode]{}.. controls (9.5459, 28.3402) and (9.7814, 28.4058) .. (10.0089, 28.6039)node[bnode]{}.. controls (10.1601, 28.4521) and (10.312, 28.3011) .. (10.5464, 28.2594)node[bnode]{}.. controls (10.4124, 28.1221) and (10.298, 27.975) .. (10.403, 27.7189)node[bnode]{}.. controls (10.2297, 27.6097) and (10.0345, 27.5318) .. (9.9572, 27.2855)node[bnode]{}.. controls (9.5287, 27.3099) and (9.2745, 27.1695) .. (8.9975, 27.0505)node[bnode]{}.. controls (8.7897, 27.2455) and (8.5451, 27.3336) .. (8.2649, 27.3177) -- cycle;

      \path[draw=black,thick,line cap=butt,line join=miter,] (8.2649, 27.3177).. controls (8.9068, 27.5757) and (9.6297, 27.6916) .. node[pos=0.3,yshift=-8]{$\alpha_0$}(10.403, 27.7189);

      \path[draw=black,thick,line cap=butt,line join=miter,] (8.5264, 28.5411).. controls (8.6749, 28.0487) and (8.5334, 27.6604) .. (8.2649, 27.3177);


      \path[draw=black,thick,line cap=butt,line join=miter,] (10.403, 27.7189).. controls (10.1624, 27.9639) and (9.9745, 28.2331) .. (10.0089, 28.6039);

      \path[draw=black,thick,line cap=butt,line join=miter,] (8.9975, 27.0505).. controls (9.6698, 27.5962) and (9.9949, 27.5918) .. (10.403, 27.7189);

      \path[draw=black,line cap=butt,line join=miter,densely dotted] (9.3265, 27.6075)node[bnode]{}node[yshift=-8]{$m$}.. controls (9.2432, 27.9339) and (9.2942, 28.3171) .. node[pos=0.3,xshift=-8]{}(9.3038, 28.3821)node[yshift=8]{$w_m$};

      \path[draw=black,line cap=butt,line join=miter,thick,miter limit=4.0] (8.2649, 27.3177).. controls (8.7551, 27.5867) and (9.1397, 27.9187) .. (9.3038, 28.3821).. controls (9.4787, 28.0831) and (9.6034, 27.7637) .. (10.403, 27.7189);
      \end{scope}
    \end{tikzpicture}
    \caption{The thick lines are disjoint $\epsilon$-geodesics. In the case where the curves crossing the short path $\beta_m$ are not connected, removing them and adding $\epsilon$-geodesics from their endpoints to $v_m$ increases the cardinality of the family}\label{fig:constructing beta_m}
  \end{figure}
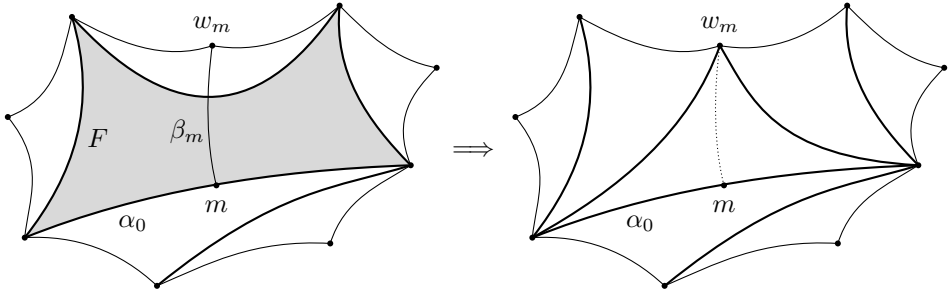

  Observe that if $m$ is chosen so that both $d(v_-,m)$ and $d(m,v_+)$ are at least $\kappa$, then $w_m$ is neither $v_-$ nor $v_+$, and is hence some other vertex in $\partial P'\smallsetminus \alpha_0$.
  Given such a point $m$, let $\beta_m$ be a piecewise smooth $\delta$-geodesic connecting $m$ and $w_m$ in $P'$, where $\delta>0$ is small enough that we may apply \cref{cor: epsilon geodesics in minimal position} to modify curves in $A\smallsetminus\{\alpha_0\}$ so that they are still $\epsilon$-geodesic and are in minimal position with respect to $\beta_m$. Moreover, by \eqref{eq:distave of w_m}, $\delta$ can be chosen small enough that   
  \begin{equation}\label{eq:lenth of beta}
    \abs{\beta_m}+\epsilon\leq (3\theta +\Theta/2+\epsilon) +\delta+\epsilon < \kappa = 3\theta +\Theta.
  \end{equation}

  Let $B_m\subseteq A$ be the set of curves that intersect $\beta_m$ outside $\partial P$.
  We observe that, as a graph, $B_m$ contains no closed loop, \emph{i.e.}\ it is a forest, and it is non-empty because $\alpha_0$ belongs to it. By Euler's formula, we thus have
  \[
    \abs{V(B_m)}=\abs{B_m} + \abs{\{\text{components of }B_m\}}\geq \abs{B_m}+1.
  \]
  Also note that if $\alpha\in B_m$ intersects $\beta_m$ in some point $x$ which splits $\alpha$ as $\alpha_-\cup \alpha_+$, then
  \labtequ{shortest intersection}{both $\abs{\alpha_-}$ and $\abs{\alpha_+}$ are at least  $d(x,w_m)$,}
  because otherwise $w_m$ would not be a closest vertex to $m$. One can hence perform the following construction:
  \labtequ{replacing curves}{
  Remove every curve in $B_m$ from $A$. In their place, add for every vertex in $V(B_m)$ which is not adjacent to $w_m$ in $\partial P'$ one arc connecting it with $w_m$.
  Then add back $\alpha_0$ and apply \cref{cor: epsilon geodesics in minimal position}, to make these curves into piecewise smooth $\epsilon$-geodesic arcs in minimal position (\cref{fig:constructing beta_m}). 
  Denote by $A_m$ be the resulting family of curves.
  }

  By \eqref{shortest intersection}, the curves in $A_m$ can be taken to have length less than $12\kappa$: if $v$ is an endpoint of $\alpha\in B_m$, it can be joined to $w_m$ by following $\alpha$ until it reaches $\beta_m$ and then following the latter until $w_m$.  
  We would like to reach a contraction by showing that $\abs{A_m}>\abs{A}$.

  We argue by cases. If $w_m\in\partial F$ then $B_m$ consisted uniquely of the curve $\alpha_0$. Since $F$ is not a triangle, at least one between $v_-$ and $v_+$ is was not adjacent to $w_m$ in $\partial F$, thus $A_m$ now has strictly more curves than $A$, so we are done.
  We may hence assume that $w_m\notin\partial F$, and hence $\abs{B_m}>1$. Since $w_m$ is adjacent to at most two vertices, we see that 
  \begin{equation}\label{eq:number of curves}
    \abs{A_m}\geq \abs{A\smallsetminus B_m} + (\abs{V(B_m)}+1)-2\geq \abs{A}.
  \end{equation}
  If $B_m$ is not connected, the above inequality is strict.

  It remains to deal with the case where $\abs{B_m}>1$ and $B_m$ is connected. This case is the most delicate, and to deal with it we may have to make different choices of $m$ (see below).

  Let $\alpha_m\in A_m$ be the first curve encountered by $\beta$ after $\alpha_0$. Since $\alpha_m$ disconnects $P'$ and $\beta$ intersects it exactly once (because they are in minimal position), we deduce that $\alpha_m$ must meet $\alpha_0$ at one of its endpoints, for otherwise $B_m$ would not be connected. Let $v_m\in\partial F$ denote the other end point of $\alpha_m$.

  If $\alpha_m\cap \alpha_0=v_+$, then $v_m$ is not adjacent to $v_-$ in $\partial F$, so if $d_{F}(v_-,v_m)$ was less than $12\kappa$ we could add another curve to $A_m$ and conclude the proof. A symmetric argument applies if $\alpha_m\cap \alpha_0=v_-$.
  This reduces to the following case: for every $m\in \alpha_0$ with $d(v_-,m),d(m,v_+)\geq \kappa$, the curve $\beta_m$ meets some $\alpha_m\subset \partial F$ connecting one of $v_-$ or $v_+$ to a point $v_m$ at $d_F$-distance at least $12\kappa$ from the other one.
  
  We claim that if $d_F(v_-,m) =\kappa$, then $\alpha_m$ must intersect $\alpha_0$ at $v_-$. 
  Suppose this is not the case, and let $x_m= \alpha_m\cap \beta_m$. Then 
  \begin{align*}
    d_F(x_m,v_+) &\geq d_F(v_+,m) - d_F(x_m,m) \\
     &\geq (\abs{\alpha_0} -\epsilon - \kappa) -\abs{\beta_m} \\
     &\geq 6\kappa - \kappa - (\abs{\beta}+\epsilon) \geq 4\kappa,
  \end{align*}
  where we used $\epsilon$-geodesicity of $\alpha_0$ at the second step.
  But then
  \begin{align*}
    d_F(v_-,v_m)&\leq d_F(v_-,x_m) + d_F(x_m,v_m) \\
      &\leq (d_F(v_-,m) + d_F(m,x_m)) + (\abs{\alpha_m} - d_F(x_m,v_+)) \\
      &\leq (\kappa + \abs{\beta}) + (12\kappa - 4\kappa) < 12\kappa.
  \end{align*}
  The symmetric holds if $d_F(v_+,m) =\kappa$.

  By continuity, this implies that there must be some point $m\in\alpha_0$ for which there are two possible choices of $w_m$, one of which yields as $\alpha_m$ a curve $\alpha_-\in A$ containing $v_-$, and the other yields as $\alpha_m$ a curve $\alpha_+\in A$ containing $v_+$. We denote those vertices $w_m^-$ and $w_m^+$ respectively, and  define $\beta_m^-$, $\beta_m^+$  similarly (\cref{fig:two beta_m}).

  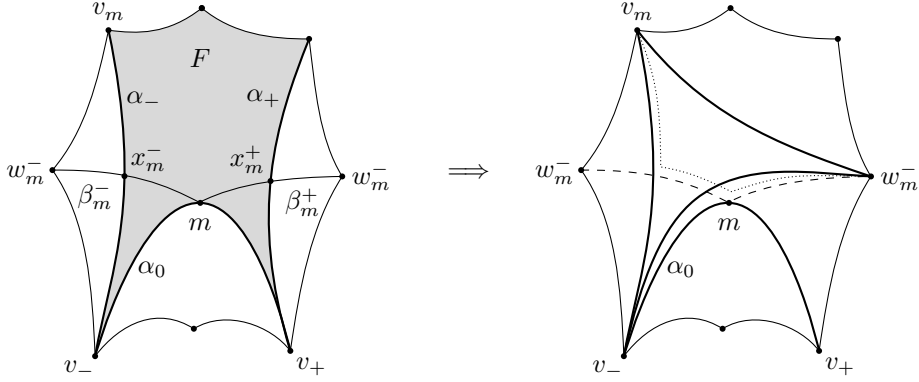
\begin{figure}
    \begin{tikzpicture}[y=1cm, x=1cm, yscale=\globalscale,xscale=\globalscale, every node/.append style={scale=1}, inner sep=0pt, outer sep=0pt]

      \fill[black!15](10.1393, 26.2547).. controls (9.88, 26.2533) and (9.6323, 26.2891) .. (9.426, 26.4597)node[bnode]{} .. controls (9.1977, 26.2815) and (9.0002, 26.2874) .. (8.8059, 26.3127)
      .. controls (9.0439, 25.1772) and (8.8248, 24.7465) .. (8.7169, 24.1443) .. controls (8.7169, 24.1443) and (9.0039, 25.1752) .. (9.4158, 25.164).. controls (9.7997, 25.1535) and (10.0146, 24.1795) ..(10.0146, 24.1795).. controls(9.7339, 25.0965) and (9.9145, 25.6911) ..  (10.1393, 26.2547);

      \path[draw=black,line cap=butt,line join=miter,]  (10.0146, 24.1795)   node[bnode]{}.. controls (10.0767, 24.6001) and (10.1488, 25.0143) .. (10.3595, 25.3399)node[bnode]{}.. controls (10.2041, 25.5906) and (10.1802, 25.9283) .. (10.1393, 26.2547)node[bnode]{}.. controls (9.88, 26.2533) and (9.6323, 26.2891) .. (9.426, 26.4597)node[bnode]{} .. controls (9.1977, 26.2815) and (9.0002, 26.2874) .. (8.8059, 26.3127)node[bnode]{}.. controls (8.7342, 25.9724) and (8.6521, 25.6382) .. (8.4331, 25.382)node[bnode]{}.. controls (8.6645, 25.0536) and (8.6977, 24.6033) .. (8.7169, 24.1443)node[bnode]{}.. controls (8.8887, 24.3695) and (9.1366, 24.478) .. (9.3732, 24.3268)node[bnode]{}.. controls (9.5913, 24.4601) and (9.8888, 24.3462) .. (10.0146, 24.1795);

      \path[draw=black,line cap=butt,line join=miter,thick] (8.7169, 24.1443).. controls (8.7169, 24.1443) and (9.0039, 25.1752) .. node[pos=0.6,yshift=-5,xshift=6]{$\alpha_0$}(9.4158, 25.164)node[bnode]{}node[yshift=-8]{$m$}node[yshift=55]{$F$}.. controls (9.7997, 25.1535) and (10.0146, 24.1795) ..(10.0146, 24.1795) ;

      \path[draw=black,line cap=butt,line join=miter,thick] (8.7169, 24.1443)node[yshift=-5,xshift=-6]{$v_-$}.. controls (8.8248, 24.7465) and (9.0439, 25.1772) .. node[pos=0.85,yshift=0,xshift=9]{$\alpha_-$} node[pos=0.65,bnode]{}node[pos=0.65,xshift=9,yshift=7]{$x_m^-$} (8.8059, 26.3127)node[yshift=9]{$v_m^-$};

      \path[draw=black,line cap=butt,line join=miter,thick] (10.1393, 26.2547).. controls (9.9145, 25.6911) and (9.7339, 25.0965) .. node[pos=0.23,yshift=0,xshift=-9]{$\alpha_+$}node[pos=0.52,bnode]{}node[xshift=-9,yshift=6]{$x_m^+$} (10.0146, 24.1795)node[yshift=-6,xshift=8]{$v_+$};

      \path[draw=black,line cap=butt,line join=miter,] (8.4331, 25.382)node[xshift=-10]{$w_m^-$}.. controls (8.9605, 25.3875) and (9.1922, 25.2797) .. node[pos=0.2,yshift=-9]{$\beta_m^-$}(9.4104, 25.1674).. controls (9.6709, 25.2994) and (10.0034, 25.335) .. node[pos=0.75,yshift=-9]{$\beta_m^+$}(10.3595, 25.3399)node[xshift=11]{$w_m^-$};

      \draw (11.2, 25.36) node{$\Longrightarrow$};

    \begin{scope}[xshift=100]
        \path[draw=black,line cap=butt,line join=miter,]  (10.0146, 24.1795)   node[bnode]{}.. controls (10.0767, 24.6001) and (10.1488, 25.0143) .. (10.3595, 25.3399)node[bnode]{}.. controls (10.2041, 25.5906) and (10.1802, 25.9283) .. (10.1393, 26.2547)node[bnode]{}.. controls (9.88, 26.2533) and (9.6323, 26.2891) .. (9.426, 26.4597)node[bnode]{} .. controls (9.1977, 26.2815) and (9.0002, 26.2874) .. (8.8059, 26.3127)node[bnode]{}.. controls (8.7342, 25.9724) and (8.6521, 25.6382) .. (8.4331, 25.382)node[bnode]{}.. controls (8.6645, 25.0536) and (8.6977, 24.6033) .. (8.7169, 24.1443)node[bnode]{}.. controls (8.8887, 24.3695) and (9.1366, 24.478) .. (9.3732, 24.3268)node[bnode]{}.. controls (9.5913, 24.4601) and (9.8888, 24.3462) .. (10.0146, 24.1795);

        \path[draw=black,line cap=butt,line join=miter,thick] (8.7169, 24.1443).. controls (8.7169, 24.1443) and (9.0039, 25.1752) .. node[pos=0.6,yshift=-5,xshift=6]{$\alpha_0$}(9.4158, 25.164)node[bnode]{}node[yshift=-8]{$m$}.. controls (9.7997, 25.1535) and (10.0146, 24.1795) ..(10.0146, 24.1795) ;

        \path[draw=black,line cap=butt,line join=miter,thick] (8.7169, 24.1443)node[yshift=-5,xshift=-6]{$v_-$}.. controls (8.8248, 24.7465) and (9.0439, 25.1772) .. (8.8059, 26.3127)node[yshift=9]{$v_m^-$};

        \path[draw=black,line cap=butt,line join=miter,thick] (10.0146, 24.1795)node[yshift=-6,xshift=8]{$v_+$};

        \path[draw=black,line cap=butt,line join=miter,dashed] (8.4331, 25.382)node[xshift=-10]{$w_m^-$}.. controls (8.9605, 25.3875) and (9.1922, 25.2797) ..(9.4104, 25.1674).. controls (9.6709, 25.2994) and (10.0034, 25.335) .. (10.3595, 25.3399)node[xshift=11]{$w_m^-$};

        \path[draw=black,line cap=butt,line join=miter,thick] (8.7169, 24.1443).. controls (8.8325, 24.6698) and (8.9775, 25.0145) .. (9.191, 25.1862).. controls (9.5156, 25.4473) and (9.9563, 25.3632) .. (10.3595, 25.3399)
        (10.3595, 25.3399).. controls (9.6953, 25.5618) and (9.232, 25.7644) .. (8.8059, 26.3127);

        \path[draw=black,line cap=butt,line join=miter,densely dotted] (8.8059, 26.3127).. controls (8.9198, 26.0344) and (8.9614, 25.8669) .. (8.9618, 25.4031).. controls (9.1626, 25.3717) and (9.3004, 25.3024) .. (9.4312, 25.2387).. controls (9.7812, 25.3553) and (9.8348, 25.352) .. (10.3595, 25.3399);
    \end{scope}

    \end{tikzpicture}
    \caption{The thick lines are disjoint $\epsilon$-geodesics. If there is a point $m$ connected to $\partial P$ via two short paths $\beta_m^\pm$ that intersect $\epsilon$-geodesics connected to the starting curve $\alpha_0$, we increase the number of disjoint $\epsilon$-geodesic by removing one of them and adding $\epsilon$-geodesics from the endpoints of the other one.}\label{fig:two beta_m}
  \end{figure}
  
  We may assume that $d(m,v_+)\leq d(m,v_-)$, which implies that 
  \[
    d(m,v_-)\geq (6\kappa-\epsilon)/2.
  \]
  We then perform the construction \eqref{replacing curves} with respect to $w_m^+$ to obtain $A_m$. Let $x_m^-$ be the intersection point of $\alpha_-$ and $\beta_m^-$, and let $v_m^-$ be the endpoint of $\alpha_-$ away from $v_-$. Observe that 
  \begin{align*}
  d_F(v_m^-,x_m^-)&\leq \abs{\alpha_m^-}-d_F(v_-,x_m^-) \\
    &\leq \abs{\alpha_m^-} - (d_F(v_-,m)-d_F(m,x_m^-))\\
    &\leq \abs{\alpha_m^-} - (6\kappa-\epsilon)/2 + \abs{\beta_m^{-}} \\
    &\leq \abs{\alpha_m^-} - 2\kappa.
  \end{align*}
  The curve following $\alpha_-$ from $v_m^-$ to $x_m^-$ and then going to $w_m$ following $\beta_m^-$ and $\beta_m^+$ has length at most
  \[
  d_F(v_m^-,x_m^-)+\abs{\beta_m^-}+\abs{\beta_m^+}\leq \abs{\alpha_m^-}.
  \]
  This means we can add one extra curve to $A_m$, which contradicts the maximality of $A$.
\end{proof}

\begin{proof}[Proof of \cref{thm: filling triangulation}]
{We already know} that $\overline G\subset \Sigma$ is a cell-decomposition (\cref{cor: cell decomposition}), and $\overline G$ is a metric graph with edge lengths bounded by $\Theta$,  and the embedding $\overline G\hookrightarrow\Sigma$ is a $1$-Lipschitz (\cref{rmk:metrizing embedded graphs}) quasi-isometry (\cref{lem: quasi isometry}).

  For every face $F$ of $\overline G\subset \Sigma$, the completion $\widehat F$ with respect to the intrinsic metric is a $\theta$-thin piecewise smooth polygon $P\coloneqq \widehat F$ with edge-length bounded by $\Theta$. Apply \cref{prop: triangulate polygons} to every such polygon to obtain a new piecewise smoothly embedded graph $G'\subset \Sigma$ whose edges have length at most $12(3\theta +\Theta) < 30\Theta$, and such that $G'\subset \Sigma$ is a cell decomposition where all the faces are $\theta$-thin polygons with boundary of size at most $3$. In particular, the intrinsic diameter of every face is bounded by $3/2\cdot 30\Theta +\theta < 46\Theta$.  
  Observe that $G'$ remains locally finite, because each face contributes finitely many new edges. Since $G'$ is obtained from $\overline{G}$ by adding extra edges, the $1$-Lipschitz embedding $G'\hookrightarrow \Sigma$ must be a quasi-isometry as well (adding new edges in this fashion can only improve the lower bound in \eqref{eq:coarse embedding}).

  Finally, we apply \cref{lem: removing 2-gons} to obtain a 3-gonal decomposition $G\subseteq G'\subset \Sigma$ so that the embedding $G'\hookrightarrow\Sigma$ is a quasi-isometry and the faces have intrinsic diameter bounded by $138\Theta$.
\end{proof}

\subsection{Constructing a quasi-isometrically embedded metric graph}\label{ssec:construction of G}
Proving \cref{thm: triangulation} essentially amounts  to exhibiting  $X\subset \overline{ G}\subset \Sigma$ satisfying the hypotheses of \cref{thm: filling triangulation}.

It is easy to see that such a set $X$ exists, for instance by choosing a $\theta$-net in $\Sigma$  (\emph{e.g.}\ picking a maximal $\theta$-separated subset) and adding to it a $\theta$-net of each connected component of $\partial \Sigma$. The resulting set is locally finite because every point in $\Sigma$ has a neighbourhood that intersects at most one component $C$ of $\partial \Sigma$ and this intersection has finite diameter in $(C,d_\ell^C)$.

We will construct a graph $\overline G$ obtained by uniting $\partial \Sigma$ with a family of curves of length at most $\Theta$ joining points in $X$. The key point is to choose sufficiently many curves to represent every (local) homotopy type of short curves.

Formally, enumerate the pairs of points $\{x,y\}\subset X$ with $d_\Sigma(x,y)<\Theta$, and for every such pair choose some compact subsurface (with boundary) $\Sigma_{\{x,y\}}\subseteq \Sigma$ that contains ${B(x;\Theta)}\cap B(y;\Theta)$ and is itself contained in a $\theta/2$-neighbourhood of it (\cref{lem:regular neighbourhood}).
Consider homotopy classes of curves $\gamma\colon [0,1]\to \Sigma_{\{x,y\}}$ joining $x$ to $y$. Observe that there are only finitely many such classes, because $\Sigma_{\{x,y\}}$ is a compact surface. Enumerate the homotopy classes that contain a curve $\gamma$ of length less than $\Theta$.

We will now iteratively choose a family $\Gamma_{\{x,y\}}$ of representatives for these homotopy classes.
Specifically, for every such homotopy class, we choose a representative $\overline \gamma$ that is a piecewise smooth curve of length less than $\Theta$ which may only meet $\partial \Sigma$ at its endpoints and intersects all the previously chosen representatives (including those belonging to $\Gamma_{\{x',y'\}}$ for previously considered pairs $\{x',y'\}\subset X$) in finitely many points. Such a representative can be found \emph{e.g.}\ by perturbing a shortest representative of the homotopy class using \cref{lem: resolve intersections}.

Let
\[
\Gamma_\Sigma\coloneqq \bigcup\bigbrace{\Gamma_{\{x,y\}}\bigmid \{x,y\}\subset X,\ d_\Sigma(x,y)<\Theta}.
\]
Since $X$ is locally finite, and each $\Gamma_{\{x,y\}}$ is a finite collection of curves, the collection $\Gamma_\Sigma$ is locally finite in $\Sigma$.

We let $\overline G$ be the graph traced by $\Gamma\cup \partial\Sigma$.
More precisely, we let 
\[
  V( \overline G)\coloneqq X\cup \bigcup\bigbrace{\overline\gamma\cap\overline\gamma'\bigmid \overline\gamma,\overline\gamma'\in \Gamma}\subset \Sigma
\]
and let $E(\overline G)$ be the set of all the subsegments of curves in $\Gamma$ or $\partial \Sigma$ that connect points in $V( \overline G)$ and do not contain any other vertex in their interior. Note that the edges coming from $\partial\Sigma$ have length bounded by $\Theta$ because  $X\cap\partial\Sigma$ is $\theta$-dense in the intrinsic metric.

{The graph $\overline{G}$} will generally have loops and multiple edges (especially if $\Theta$ is large). By construction, $\overline G$ is embedded in $\Sigma$. 

{
\begin{remark}\label{rmk: no extra vertices on boundary}
  Observe that, by construction, $V(\overline G)\cap \partial \Sigma = X\cap \partial \Sigma$.
\end{remark}
}
\medskip

It only remains to verify \eqref{homotopy approximation}.  This is the content of the next lemma.

\begin{lem}\label{lem:close homotopy}
  Let $\gamma\colon [0,1]\to \Sigma$ be a rectifiable curve with endpoints in $X\subseteq V(\overline G)$. Then there exists a path $\bar\gamma$ in $\overline G\subset \Sigma$ such that
  \begin{enumerate}
    \item $\gamma$ and $\bar\gamma$ are homotopic within the $2\Theta$-neighbourhood of $\gamma$;
    \item $\abs{\bar\gamma}$ satisfies
      \[
      \abs{\bar\gamma}\leq \frac{\Theta}{\Theta-2\theta}\abs{\gamma} + \Theta.
      \]
  \end{enumerate}
\end{lem}
\begin{proof}
  Parameterise $\gamma$ by arc length.
  Let $n= \lfloor \abs{\gamma}/(\Theta-2\theta)\rfloor$ and for every $0\leq k \leq n$ let $p_k \coloneqq\gamma(k(\Theta -2\theta))$. Finally, let $p_{n+1}\coloneqq y$.
  For every $0\leq k\leq n+1$ let $x_k$ be a point in $X$ nearest to $p_k$ (in particular, $x_0=p_0$ and $x_{n+1}=p_{n+1}$). Further let $\alpha_k$ be a geodesic path from $p_k$ to $x_k$, and note that $\abs{ \alpha_k} \leq \theta$ since $X$ is $\theta$-dense in $\Sigma$.

  Let $\gamma_k$ denote the segment of $\gamma$ between $p_{k}$ and $p_{k+1}$. The curve $\alpha_{k}^{-1}\gamma_k\alpha_{k+1}$ has length at most $\Theta$. 
  By construction, $\alpha_{k}^{-1}\gamma_k\alpha_{k+1}$ is contained in the subsurface $\Sigma_{\{x_k,x_{k+1}\}}$, and is hence homotopic within $\Sigma_{\{x_k,x_{k+1}\}}$ to one of the fixed representatives $\overline\gamma_k\in\Gamma_{\{x_k,x_{k+1}\}}$. In particular, the homotopy between $\alpha_{k}^{-1}\gamma_k\alpha_{k+1}$ and $\overline\gamma_k$
  is entirely contained in the $(\Theta+\theta)$-ball centered at $p_k$.
  By construction, $\bar\gamma_k$ is contained in $\overline G$. 
  
  We may then join the curves $\bar\gamma_k$ to obtain a curve $\bar \gamma$ that is homotopic to $\gamma$. Moreover, we have
  \[
  \abs{\bar \gamma} = \Sigma_{k<n} \abs{\bar \gamma_k} \leq (n+1)\Theta \leq \frac{\Theta}{\Theta-2\theta}\abs{\gamma} + \Theta.\qedhere
  \]
\end{proof}

\begin{proof}[Proof of \cref{thm: triangulation}]
  Make the above construction choosing a $\Theta\leq \Xi/138$ and apply \cref{thm: filling triangulation}.
\end{proof}

\begin{remark}\label{rmk:tame surfaces have bounded degrees}
  If $\Sigma$ is a surface where the surfaces $\Sigma_{\{x,y\}}$ can be chosen to be simply connected (\emph{e.g.}\ because they are convex), and $X$ is a uniform net, then the graph $\overline G$ has bounded degree.
\end{remark}

\section{Achieving $V(G)=X$, and comparison to Maillot's results}

In \cref{thm: triangulation} we started our construction of $G$ by picking a net $X$ of $\Sigma$ which became a subset of $V(G)$, but we also added further vertices to $G$. Maillot \cite{Maillot} used a similar construction in order to prove that virtual surface groups are exactly the groups quasi-isometric to a complete simply-connected Riemannian surface. In \cite{Maillot}, it was important that the vertex set of the 3-gonal decomposition $G$ coincides with a fixed net. The aim of this section is to adapt our above construction to achieve this restriction. We will thereby improve on some of Maillot's results.

\medskip

The \emph{modulus} of a cell decomposition $G\subset \Sigma$ is defined to be the supremum of the length of the $1$-cells. In particular, if $G\hookrightarrow \Sigma$ is a $1$-Lipschitz embedded metric graph with edge lengths bounded by $\Theta$ and defining a cell decomposition, then the decomposition has modulus bounded by $\Theta$.

\begin{thm} \label{thm Maillot}
Let $(\Sigma, d_\Sigma)$ be a complete Riemannian surface possibly with boundary and corners. For each net $X$ in $(\Sigma, d_\Sigma)$ such that $X\cap \partial \Sigma$ is a net in each component of $\partial\Sigma$ with its intrinsic metric, there exists a 3-gonal decomposition $\cd$ of $\Sigma$ that has finite modulus and whose $0$-skeleton is exactly $X$.

Moreover, if $\partial \Sigma$ satisfies \eqref{comparable metric} and $X$ is uniform, then $\cd$ can be chosen so that its $1$-skeleton equipped with the simplicial metric is quasi-isometric to $(\Sigma, d_\Sigma)$.
\end{thm}

If $\partial \Sigma=\emptyset$, the first statement is \cite{Maillot}*{Theorem 4.4} (called the `main technical
result' of that paper). The analogous fact with $\partial\Sigma\neq \emptyset$ is used in \cite{Maillot}*{Theorem 8.1}, where it is remarked that their proof of \cite{Maillot}*{Theorem 4.4} works in this case as well as in the boundary-free case, but details are missing.

The second statement bypasses Proposition 4.5 and Lemma~5.2 of \cite{Maillot} without assuming the additional conditions of planarity at infinity or lower-bounded curvature used there.

\begin{proof}
Suppose that $X$ is a $\theta$-dense net such that $X\cap \partial \Sigma$ is $\theta$-dense in each component with its intrinsic metric. We use this $X$ in the construction of \cref{ssec:construction of G} with some fixed $\Theta> 2\theta$.
Then \cref{thm: triangulation} yields a $3$-gonal decomposition $G\subseteq\Sigma$ of finite modulus with $X\subseteq V(G)$ and such that $(G,d_G)\hookrightarrow (\Sigma,d_\Sigma)$ is a 1-Lipschitz quasi-isometry.

Let $\cv= \{V_x \mid x\in X\}$ be the Voronoi decomposition of $V(G)$ with respect to $X$ in the metric $d_G$, \emph{i.e.}\ $V_x\coloneqq \{v \in V(G) \mid d_G(v,X)=d_G(v,x)\}$. By perturbing $d_G$ slightly we can break all ties, \emph{i.e.}\ ensure that the $V_x$ are pairwise disjoint. 

Let $G_x$ be the subgraph of $G$ induced by the vertices in $V_x$. Note that each $G_x$ is connected. Moreover, we have
\labtequ{Vi diam}{$\diam(G_x)\leq \Theta$ for every $x\in X$,}
because each vertex of $G$ has distance at most $\Theta/2$ from $X$.

For each $x\in X$ we pick a \emph{geodetic spanning tree} \emph{$T_x$} of $V_x$.\footnote{
  This means that for each $y\in V_x$ we have $d_{V_x}(x,y)=d_{T_x}(x,y)$. Such a $T_x$ can be obtained by ordering the vertices of $V_x$ according to their distance from $x$, and recursively joining the next closest vertex to the tree constructed so far.
  }
{Note that $T_x\cap \partial\Sigma$ is either empty or the single point $x$ (if $x\in X\cap \partial \Sigma$). This is because at no point of the construction did we add extra vertices to $\partial \Sigma$ (\cref{rmk: no extra vertices on boundary}). 
}

We can further choose a surface $N_x\subset\Sigma$ contained in an $\epsilon$-neighbourhood of $T_x$ so that $G\cap N_x$ is a tree (this has the effect of adding to $T_x$ a final segment of each edge $e\in G$ which ends in $T_x$ without being contained in it). If $x\in \partial\Sigma$, this tree will also contain a small segment in $\partial\Sigma$, which we do not wish to alter. Let then $T'_x$ be the tree $G\cap N_x$ excluding the edges in $\partial \Sigma$ (if present).
We can now apply \cref{lem: resolve intersections} to the tree $T'_x$  rooted at $x$ in the surface $N_x$, in such a way that $\partial N_x\cap T'_x$ is preserved. For each edge $e$ not contained in $T_x$ nor in $\partial \Sigma$ but having an end-vertex in $T_x$, we then extend $e \smallsetminus N_x$ by the $x$--$\partial N_x$~curve returned by \cref{lem: resolve intersections} ending at $e\cap \partial N_x$. 
Pictorially, this has the effect of `dragging' $T_x$ onto $x$ inside a small enough neighbourhood $N_x$ of $T_x$ in $\Sigma$, pulling all edges incident with $T_x$ as we drag (\cref{fig PT}).

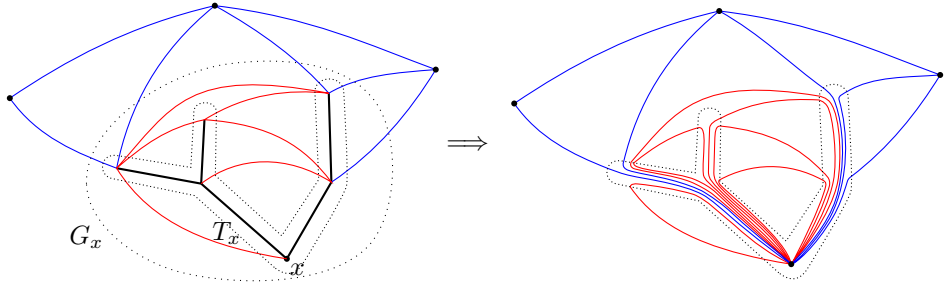
\begin{figure}
    \begin{subfigure}[m]{0.47\textwidth}
      \centering
      \begin{tikzpicture}[y=1cm, x=1cm, yscale=\globalscale*.7,xscale=\globalscale*.7, every node/.append style={scale=1}, inner sep=0pt, outer sep=0pt]
        \path[draw=black,line cap=butt,line join=miter,thick] (3.7597, 18.1762)node[bnode]{}node[xshift=4, yshift=-4]{$x$} -- (4.1829, 18.9045) -- (4.1582, 19.7497)(2.9464, 18.8904) -- (2.9775, 19.498)(3.7597, 18.1762) -- node[pos=0.7,yshift=-10]{$T_x$}(2.9464, 18.8904) -- (2.1437, 19.0372);
        \path[draw=blue,line cap=butt,line join=miter,] (5.1741, 19.9743)node[bnode]{}.. controls (4.8095, 19.9466) and (4.4437, 19.921) .. (4.1582, 19.7497)(3.076, 20.5819)node[bnode]{}.. controls (3.9865, 20.4743) and (4.5422, 20.2071) .. (5.1741, 19.9743).. controls (4.9217, 19.565) and (4.6141, 19.193) .. (4.1829, 18.9045)(2.1437, 19.0372).. controls (1.7672, 19.1824) and (1.3921, 19.3303) .. (1.1293, 19.703)node[bnode]{}.. controls (1.6667, 20.0647) and (2.2599, 20.3921) .. (3.076, 20.5819).. controls (2.5331, 20.1422) and (2.2591, 19.6154) .. (2.1437, 19.0372)(3.076, 20.5819).. controls (3.515, 20.3391) and (3.882, 20.0645) .. (4.1582, 19.7497);
        \path[draw=red,line cap=butt,line join=miter,] (2.1437, 19.0372).. controls (2.473, 18.4981) and (3.0715, 18.2832) .. (3.7597, 18.1762)(2.9775, 19.498).. controls (3.3785, 19.7482) and (3.7746, 19.7886) .. (4.1582, 19.7497).. controls (2.8222, 20.0004) and (2.6736, 19.633) .. (2.1437, 19.0372)(2.1437, 19.0372).. controls (2.3992, 19.2558) and (2.6141, 19.4199) .. (2.9775, 19.498).. controls (3.6148, 19.4555) and (3.9278, 19.2425) .. (4.1829, 18.9045).. controls (3.7642, 19.1693) and (3.2998, 19.1603) .. (2.9464, 18.8904);
        \path[draw=black,line cap=butt,line join=miter,,miter limit=4.0,dotted] (1.8573, 18.8896)node[yshift=-20]{$G_x$}.. controls (1.8408, 19.4559) and (2.6319, 19.9492) .. (3.193, 20.0268).. controls (3.6975, 20.0965) and (4.3339, 20.0589) .. (4.6086, 19.6558).. controls (4.8393, 19.3175) and (4.8138, 18.5869) .. (4.5289, 18.2928).. controls (4.1572, 17.9089) and (3.271, 17.9184) .. (2.7526, 18.0483).. controls (2.3546, 18.1481) and (1.8692, 18.4794) .. (1.8573, 18.8896) -- cycle;
        \path[draw=black,line cap=butt,line join=miter,densely dotted,miter limit=4.0] (3.6412, 18.0943) -- (2.8843, 18.7768) -- (2.1081, 18.9373).. controls (1.9635, 18.9616) and (1.9811, 19.1757) .. (2.1404, 19.1586) -- (2.848, 19.0446) -- (2.8771, 19.5616).. controls (2.8674, 19.6911) and (3.0892, 19.6908) .. (3.0911, 19.5687) -- (3.0811, 18.968) -- (3.7113, 18.3954) -- (4.0627, 18.9293) -- (4.0266, 19.7566).. controls (4.0201, 19.9149) and (4.2696, 19.9317) .. (4.299, 19.7796).. controls (4.299, 19.7796) and (4.3487, 18.9844) .. (4.3272, 18.8729).. controls (4.3055, 18.76) and (3.9084, 18.1254) .. (3.9084, 18.1254).. controls (3.8343, 18.0108) and (3.7409, 18.0143) .. (3.6412, 18.0943) -- cycle;
      \end{tikzpicture}    
    \end{subfigure}%
    \begin{subfigure}[m]{0.06\textwidth}
      $\Longrightarrow$
    \end{subfigure}%
    \begin{subfigure}[m]{0.47\textwidth}
      \centering
      \begin{tikzpicture}[y=1cm, x=1cm, yscale=\globalscale*.7,xscale=\globalscale*.7, every node/.append style={scale=1}, inner sep=0pt, outer sep=0pt]
        \path[draw=red,line cap=butt,line join=miter,,miter limit=4.0] (8.5345, 18.2863).. controls (8.5859, 18.3694) and (8.8216, 18.8425) .. (8.8617, 19.028).. controls (8.8709, 19.0704) and (8.8449, 19.0783) .. (8.8324, 19.0852).. controls (8.5082, 19.2627) and (8.1514, 19.2482) .. (7.8571, 19.0887).. controls (7.8425, 19.0839) and (7.8406, 19.0617) .. (7.8626, 19.0376).. controls (7.9429, 18.9497) and (8.4948, 18.5158) .. (8.5345, 18.2863)(8.5345, 18.2863).. controls (8.4451, 18.3002) and (8.3572, 18.316) .. (8.2713, 18.334).. controls (7.7518, 18.4429) and (7.2812, 18.6204) .. (7.0085, 19.0165).. controls (7.0039, 19.0232) and (7.0035, 19.0562) .. (7.08, 19.0484).. controls (7.2455, 19.0317) and (7.4431, 18.998) .. (7.6037, 18.9505).. controls (7.7283, 18.9057) and (7.8684, 18.7576) .. (7.9971, 18.6556).. controls (8.1732, 18.4947) and (8.3462, 18.3701) .. (8.5345, 18.2863)(8.5345, 18.2863).. controls (8.61, 18.3624) and (8.9139, 18.7571) .. (8.9752, 19.0654).. controls (9.0108, 19.2442) and (9.0096, 19.6806) .. (8.9061, 19.8214).. controls (8.8686, 19.8723) and (8.8051, 19.8827) .. (8.7984, 19.8838).. controls (7.6659, 20.0716) and (7.4552, 19.7677) .. (7.0152, 19.2579).. controls (7.0043, 19.2349) and (7.0037, 19.2033) .. (7.0513, 19.1908).. controls (7.2275, 19.1465) and (7.4187, 19.1315) .. (7.5798, 19.068).. controls (7.6374, 19.0452) and (7.7463, 18.9866) .. (7.7968, 18.9446).. controls (8.0587, 18.7264) and (8.2802, 18.5164) .. (8.5345, 18.2863)(8.5345, 18.2863).. controls (8.3074, 18.5222) and (8.062, 18.7362) .. (7.8257, 18.9612).. controls (7.7693, 19.0148) and (7.7012, 19.0543) .. (7.6242, 19.0883).. controls (7.4478, 19.1663) and (6.9975, 19.1768) .. (7.0424, 19.2508).. controls (7.2251, 19.3985) and (7.4027, 19.5135) .. (7.6514, 19.5833).. controls (7.7505, 19.6065) and (7.6839, 19.2891) .. (7.7001, 19.142).. controls (7.7008, 19.0918) and (7.7848, 19.0292) .. (7.8319, 18.9842).. controls (8.0661, 18.7516) and (8.3275, 18.5497) .. (8.5345, 18.2863)(8.5345, 18.2863).. controls (8.5998, 18.3669) and (8.8625, 18.715) .. (8.9345, 19.0761).. controls (8.9747, 19.2779) and (8.9355, 19.6974) .. (8.875, 19.7903).. controls (8.8447, 19.8367) and (8.7948, 19.847) .. (8.7948, 19.847).. controls (8.4901, 19.862) and (8.1844, 19.8414) .. (7.8689, 19.6752).. controls (7.813, 19.6531) and (7.7592, 19.6205) .. (7.7594, 19.5755).. controls (7.7536, 19.4471) and (7.7513, 19.3041) .. (7.7484, 19.1655).. controls (7.7473, 19.0899) and (7.8008, 19.0487) .. (7.8411, 19.0047).. controls (8.0723, 18.7652) and (8.3417, 18.5984) .. (8.5345, 18.2863)(8.5345, 18.2863).. controls (8.7332, 18.5679) and (8.8744, 18.9429) .. (8.8956, 19.0251).. controls (8.9161, 19.1043) and (8.8525, 19.1428) .. (8.8371, 19.1594).. controls (8.6178, 19.3947) and (8.3378, 19.5484) .. (7.8672, 19.5983).. controls (7.8127, 19.6068) and (7.811, 19.5328) .. (7.8085, 19.4647) -- (7.7994, 19.1544).. controls (7.8015, 19.0912) and (7.8238, 19.0611) .. (7.8525, 19.0249).. controls (8.0798, 18.7787) and (8.3822, 18.604) .. (8.5345, 18.2863);    
        \path[draw=blue,line cap=butt,line join=miter,,miter limit=4.0] (9.949, 20.0844)node[bnode]{}.. controls (9.6401, 20.061) and (9.3115, 20.0759) .. (9.0693, 19.9288).. controls (9.0403, 19.9111) and (8.9763, 19.8901) .. (8.9885, 19.8487).. controls (9.0774, 19.5452) and (9.0874, 19.0944) .. (8.9094, 18.7432).. controls (8.8353, 18.5968) and (8.607, 18.3112) .. (8.5345, 18.2863)(7.8509, 20.692)node[bnode]{}.. controls (8.7614, 20.5844) and (9.3171, 20.3173) .. (9.949, 20.0844).. controls (9.7263, 19.7234) and (9.4607, 19.3914) .. (9.1052, 19.1201).. controls (9.0991, 19.1155) and (9.0767, 19.1033) .. (9.0676, 19.0732).. controls (9.0284, 18.9422) and (8.9877, 18.8321) .. (8.9306, 18.7341).. controls (8.7658, 18.451) and (8.5656, 18.2811) .. (8.5345, 18.2863)node[bnode]{}(8.5345, 18.2863).. controls (8.5151, 18.2938) and (8.1753, 18.5466) .. (7.9539, 18.7415).. controls (7.8932, 18.7949) and (7.7662, 18.9086) .. (7.6939, 18.9466).. controls (7.5905, 19.0009) and (7.4868, 19.0264) .. (7.3923, 19.05).. controls (7.2341, 19.0895) and (7.0914, 19.1069) .. (6.9844, 19.1327).. controls (6.9076, 19.1512) and (6.7939, 19.1959) .. (6.7875, 19.1985).. controls (6.4564, 19.33) and (6.1364, 19.4838) .. (5.9041, 19.8132)node[bnode]{}.. controls (6.4416, 20.1749) and (7.0347, 20.5022) .. (7.8509, 20.692).. controls (7.3455, 20.2828) and (7.0731, 19.798) .. (6.9448, 19.2666).. controls (6.9432, 19.2598) and (6.9151, 19.2034) .. (6.989, 19.1778).. controls (7.0635, 19.152) and (7.222, 19.1311) .. (7.4026, 19.0838).. controls (7.5, 19.0583) and (7.6484, 19.008) .. (7.7088, 18.9695).. controls (7.8164, 18.9011) and (7.8817, 18.8559) .. (7.9757, 18.7683).. controls (8.2357, 18.5258) and (8.5377, 18.3019) .. (8.5345, 18.2863)(7.8509, 20.692).. controls (8.2436, 20.4749) and (8.5324, 20.1784) .. (8.8427, 19.9581).. controls (8.8496, 19.9532) and (8.9136, 19.9075) .. (8.9384, 19.8521).. controls (8.9922, 19.7316) and (9.0701, 19.3925) .. (8.9811, 18.987).. controls (8.9308, 18.7578) and (8.7726, 18.5262) .. (8.5345, 18.2863);
        \path[draw=black,line cap=butt,line join=miter,,miter limit=4.0,densely dotted] (8.4161, 18.2045) -- (7.6591, 18.887) -- (6.8829, 19.0474).. controls (6.7383, 19.0717) and (6.756, 19.2859) .. (6.9152, 19.2688) -- (7.6229, 19.1548) -- (7.6519, 19.6717).. controls (7.6423, 19.8013) and (7.864, 19.8009) .. (7.866, 19.6788) -- (7.856, 19.0782) -- (8.4862, 18.5056) -- (8.8375, 19.0395) -- (8.8014, 19.8667).. controls (8.795, 20.025) and (9.0444, 20.0418) .. (9.0738, 19.8898).. controls (9.0738, 19.8898) and (9.1236, 19.0945) .. (9.1021, 18.983).. controls (9.0803, 18.8701) and (8.6833, 18.2356) .. (8.6833, 18.2356).. controls (8.6091, 18.1209) and (8.5158, 18.1244) .. (8.4161, 18.2045) -- cycle;
      \end{tikzpicture}    
    \end{subfigure}
   \caption{\small Contracting the spanning tree $T_x$ to the point $x$. The subgraph $G_x\subset G$ is the union of $T_x$ with the red edges in the left hand side.}
   \label{fig PT}
\end{figure} 

By applying this procedure to each $x\in X$, we modify $G$ into a graph $\widetilde G$ embedded into $\Sigma$ with $V(\widetilde G)=X$. Note that there is an one-to-one correspondence between the faces of $G$ and the faces of $\widetilde G$. 
The latter generally fails to be a 3-gonal decomposition, because contracting the edges of the $T_x$ will decrease the boundary size of some of the faces.
However, we may apply \cref{lem: removing 2-gons} to obtain a 3-gonal decomposition $\cd$ of $\Sigma$. Observe that
\labtequ{upper bounded}{
the lengths of the edges in $\widetilde G$ are uniformly bounded above.
}
In fact, by construction every edge in $\widetilde G$ has length within $\epsilon$ from the length of a path consisting of the concatenation an edge in $G$ and two branches in some spanning trees $T_x$ (\cref{lem: resolve intersections}), and the latter is bounded because of \eqref{Vi diam}.
This completes the proof of the first statement. (Note that at this point it is not yet clear whether the embedding $\widetilde G\hookrightarrow\Sigma$ is a quasi-isometry.)

\medskip

For the second statement, let $\widetilde G^{(1)}$ denote the $1$-skeleton with its simplicial metric, \emph{i.e.}\ the graph $\widetilde{G}$ where every edge is given length one. Observe that the lengths of the non-loop edges in $\widetilde G$ are uniformly bounded below because they join distinct points in $X$, which is a net. Since the length of all the edges is uniformly bounded above by \eqref{upper bounded}, we see that 
\labtequ{QI to 1-skeleton}{$\widetilde G^{(1)}$ and $\widetilde G\subset\Sigma$ are quasi-isometric via the identity map.}
{Note that if we prove that $\widetilde{G}^{(1)}\hookrightarrow \Sigma$ is a quasi-isometry then also the 3-gonal decomposition $\cd$ obtained by applying \cref{lem: removing 2-gons} is quasi-isometrically embedded.}

By construction, the embedding $\widetilde{G}\hookrightarrow \Sigma$ is a $1$-Lipschitz map which is coarsely surjective. Together with \eqref{QI to 1-skeleton}, it is thus enough to show that if the net $X$ is uniform, then there are constants $M,A$ such that for every $x,y\in X$ we have a bound $d_{\widetilde G^{(1)}}(x,y) \leq M d_\Sigma(x,y)+A$.

Observe that contracting each $G_x$ into a vertex we obtain a minor $H$ of $G$, which we equip with its simplicial metric.
Note that, as abstract graphs, $H$ is a subgraph of $\widetilde G^{(1)}$ obtained by removing loops and some parallel edges. In particular, the identity map on $X$ defines a quasi-isometry between $H$ and $ \widetilde G^{(1)}$.

Given $x,y \in X$, let $\gamma$ be a $x$--$y$~geodesic in $\Sigma$, and $\bar\gamma$ a $x$--$y$ geodesic in $G$. \cref{lem:close homotopy}, gives a uniform affine upper bound on $\abs{\bar\gamma}$ in terms of $d_\Sigma(x,y)$. Moreover, $\bar\gamma$ induces a $x$--$y$~path $p$ in $H$, and we have $|p|\leq C |\bar\gamma|$ for some constant $C$, because each unit of length of $\bar\gamma$ meets a uniformly bounded number of cells $V_i$ by the uniformity of $X$. Thus $d_H(x,y) \leq |p|\leq C |\bar\gamma|$ has a uniform affine upper bound in terms of $d_\Sigma(x,y)$. The claim follows.
\end{proof}

The \emph{essential degree} of a vertex $x$ of a graph $G$ is the number of $y\in V(G)$ such that $G$ contains an $x$--$y$~edge (this is smaller than the degree of $x$ if $G$ has several $x$--$y$~edges). The following refines \Cref{thm: uni net intro}.
\begin{corollary} \label{cor uni net}
The following statements are equivalent for a complete Riemannian surface $(\Sigma, d_\Sigma)$:
\begin{enumerate}
 \item \label{c i} $\Sigma$ is quasi-isometric to a graph of bounded degree;
  \item \label{c ii} $\Sigma$ has a uniform net; 
   \item \label{c iii} $\Sigma$ has a 3-gonal decomposition of bounded essential degree whose $1$-skeleton with the simplicial metric is quasi-isometric to $\Sigma$. 
\end{enumerate}
\end{corollary}
\begin{proof}
The equivalence of \eqref{c i} and \eqref{c ii} is well-known, and holds in the greater generality where $\Sigma$ is a geodesic metric space \cite{KanaiRough2}. It is easy to deduce \eqref{c i} from \eqref{c iii} by removing loops and all but one $x$--$y$~edges for each pair of adjacent vertices $x,y$. The implication \eqref{c ii} $\implies$ \eqref{c iii} is established by Theorem~\ref{thm Maillot}: the fact that \cd\ is of bounded essential degree follows from the fact that all neighbours of $x\in X=V(\cd)$ in the 1-skeleton of \cd\ are within bounded distance from $x$, and so their cardinality is bounded by the fact that $X$ is uniform.
\end{proof}

Using Remark~\ref{rmk:barycentric subdivision of 3-gonal}, we can modify the 3-gonal decomposition of \cref{cor uni net}\eqref{c iii} into a triangulation by subdividing edges and adding vertices as needed. This suffices to complete the proof of \Cref{thm: uni net intro}. (Note however that the essential degree may become unbounded in this process.)

\section{Concluding remarks and further questions}

\subsection{$(1,A)$-quasi-isometric triangulations}

In this section we record the proof of \cref{cor M=1 intro}, which was communicated to us by James Davies. The precise statement we prove is the following.
\begin{prop}
  Let $(\Sigma,d_\Sigma)$ be a complete Riemannian surface without boundary. Then there is a $3$-gonal decomposition $G\subset \Sigma$ as in \cref{thm: triangulation} and a constant $A\geq 0$, such that the embedding $(G, d_\ell^G)\hookrightarrow(\Sigma,d_\Sigma)$ is a $(1,A)$-quasi-isometry.
\end{prop}
\begin{proof}
  The idea is to choose a sufficiently dense locally finite set $X$ in the construction of \cref{ssec:construction of G}.
  Arbitrarily pick sufficiently small $0<\theta<\Theta/4$ (\emph{e.g.}\ $\Theta = \Xi/138$ and $\theta=\Xi/600$, where $\Xi$ is the desired size of triangles in \cref{thm: triangulation}), and cover $\Sigma$ with a locally finite countable family of subsets $U_n$ of diameter less than $\Theta/2$.
  In each $U_n$ we choose a finite set of points $X_n$ that is $\theta/2^n$-dense. Let $X\coloneqq \bigcup_{n\in\mathbb N}X_n$. This is a locally finite, $\theta$-dense subset of $\Sigma$, so it can be used in the construction of the embedded graph $\overline{G}$ as in \cref{ssec:construction of G}.

  The key point now is that the set $X$ is sufficiently dense to improve the strategy of proof of \cref{lem:close homotopy} to obtain multiplicative constant one. Namely, let $\gamma$ be a geodesic connecting two points $x,y\in X$, parameterised by arc length. Let $n= \lfloor \abs{\gamma}/(\Theta/2)\rfloor$, and for every $0\leq k \leq n$,  let $p_k \coloneqq\gamma(k\Theta/2)$, and let $p_{n+1}\coloneqq y$.

  For every $0\leq k\leq n+1$, let $x_k$ be a point in $X$ nearest to $p_k$. In particular, if $p_k\in U_{n_k}$, then $d_\Sigma(x_k,p_k)\leq 2^{-n_k}$, hence 
  \[
  d_\ell^{\overline{G}}(x_k,x_{k+1})=d_\Sigma(x_k,x_{k+1}) \leq 2^{-n_k} + d_\Sigma(p_k,p_{k+1}) + 2^{-n_{k+1}}
  \]
  Crucially, the points $p_k$ all belong to different sets $U_k$, because the latter have diameter less than $\Theta/2$ by assumption. By the triangle inequality, we obtain
  \[
  d_\ell^{\overline{G}}(x,y)
  \leq\sum_{k=1}^n d_\ell^{\overline{G}}(x_k,x_{k+1})
  \leq d_\Sigma(x,y) + 2\sum_{n\in\mathbb N} 2^{-n}.
  \]
  
  This completes the proof, because the other steps in the proof of \cref{thm: triangulation} do not change the multiplicative constant of the embedding $G\hookrightarrow \Sigma$; indeed, $G$ is obtained by first adding edges to $\overline{G}$ and then removing redundancies with \cref{lem: removing 2-gons}. (This last step could be problematic if $\partial\Sigma\neq \emptyset$, but since we are assuming that this is not the case, the proof of \cref{lem: removing 2-gons} shows that the multiplicative constant does not worsen.)
\end{proof}

\subsection{Quasi-isometric planar graphs that are not bi-Lipschitz equivalent} \label{sec bLe}

In this section we will combine Theorem~\ref{thm Maillot} with a construction of Burago \& Kleiner \cite{BuKlSep} in order to prove \Cref{cor AD intro}, which we restate for convenience:

\begin{corollary} \label{cor AD}
There are plane graphs $H_1,H_2$, with bounded degrees and face-boundary sizes, which are quasi-isometric to each other but not bi-Lipschitz equivalent.
\end{corollary}

\begin{proof}
Burago \& Kleiner \cite{BuKlSep} constructed a net $X_1$ in $\R^2$ which is not bi-Lipschitz equivalent to the `integer' net $X_2\coloneqq\Z^2$. It is not hard to see that every net in $\R^2$ is uniform by a volume argument. Thus, we can apply Theorem~\ref{thm Maillot} to obtain 3-gonal decompositions $\cd_1,\cd_2$ of $\R^2$ with vertex sets $X_1$, $X_2$ and 1-skeletons $H_1,H_2$ quasi-isometric to $\R^2$ with their simplicial graph metric. (For $H_2$ we could also just use the standard Cayley graph of $\Z^2$.) Thus, $H_1,H_2$ are quasi-isometric to each other. Easily, they have face-boundary sizes at most 3, and bounded degrees since $X_i$ is uniform (\cref{rmk:tame surfaces have bounded degrees}).

By \eqref{QI to 1-skeleton}, the quasi-isometry to $\R^2$ is defined by the identity $(X_i,d_{H_i}) \to (X_i,d_{\R^2})$. Since this is a bijection, we deduce that $(X_i,d_{H_i})$ is bi-Lipschitz equivalent to $(X_i,d_{\R^2})$ for $i=1,2$. Thus if $H_1,H_2$ are bi-Lipschitz equivalent, then so are $(X_1,d_{\R^2}), (X_2,d_{\R^2})$, a contradiction.
\end{proof}

\subsection{Flip-graphs of triangulations of infinite type}\label{ssec: flip-graph}
Let $\Sigma_{g,n}$ be a closed surface of genus $g$ with $n$ marked points. An \emph{ideal triangulation} of it is a $3$-gonal decomposition whose vertex set coincides with the $n$ marked points.\footnote{
  The naming comes from seeing the marked points as punctures, which are classically metrized as cusps. This is not the point of view we are going to take here though.
}
A \emph{flip} on a triangulation $D$ is a modification of $D$ obtained by choosing a quadrilateral in $D$---which must hence contain exactly one diagonal---and replacing its diagonal by a curve joining the other two vertices of the quadrilateral. This yields another ideal triangulation, uniquely defined up to \emph{marked isotopy} (\emph{i.e.}\ an isotopy that keeps the marked points fixed at all times).

A very interesting object associated with such a surface $\Sigma_{g,n}$ is its \emph{flip-graph} $\mathcal F(\Sigma_{g,n})$. This is a graph whose vertices are the marked-isotopy equivalence classes of ideal triangulations of $\Sigma_{g,n}$, and where two classes are joined by an edge if they (have representatives that) only differ by a flip. This graph is connected, and an important reason for studying it is that its group of graph-theoretic automorphisms is isomorphic to the mapping class group of $\Sigma_{g,n}$ \cite{korkmaz2012ideal}.

If one wishes to work with non compact surfaces, the situation becomes considerably more complicated. Let $X\subset \Sigma$ be a fixed set of marked points. One can analogously define a flip-graph $\mathcal F(\Sigma_X)$, but very simple examples show that if $\Sigma$ is non compact or $X$ is infinite then $\mathcal F(\Sigma_X)$ is not connected. One first remedy is to add edges to $\mathcal F(\Sigma_X)$ by declaring that two ideal triangulations are connected by an edge if they differ by a family of flips that can be performed simultaneously (\emph{i.e.}\ so that no incident edges are flipped). But this does not suffice: it is proven in \cite{fossas2022flip} that this graph remains disconnected, and this is used in \cite{bar2023big} to show that its automorphism group is strictly larger than the mapping class group.

On the other hand, \cite{fossas2022flip} also prove that two ideal triangulations $D$, $D'$ belong to the same connected component of $\mathcal F(\Sigma_X)$ if and only every edge of $D$ intersects boundedly many edges of $D'$ and vice versa.

Suppose now that $(\Sigma,d_\Sigma)$ is a complete Riemannian surface containing a uniform net $X$. Let $\mathcal F(\Sigma_X,d_\Sigma)$ be the graph having as vertices ideal triangulations\footnote{
  As above, by an \emph{ideal triangulation} we mean a 3-gonal decomposition of $\Sigma$ the vertex set of which coincides with $X$.}
that are quasi-isometric to $(\Sigma,d_\Sigma)$ when equipped with their simplicial metric (they exist by \cref{thm Maillot}) considered up to marked isotopy, and where two ideal triangulations are joined by an edge if they differ by a family of simultaneous flips. It then follows from the above criterion that $\mathcal F(\Sigma_X,d_\Sigma)$ is connected. It is natural to ask what the group of automorphisms of $\mathcal F(\Sigma_X,d_\Sigma)$ is.
Note that not every marked homeomorphism $\phi$ of $\Sigma$ defines an automorphism of $\mathcal F(\Sigma_X,d_\Sigma)$, because we are only admitting quasi-isometric triangulations. On the other hand, $\phi$ will induce an automorphism if it is also a quasi-isometry. Does every automorphism of $\mathcal F(\Sigma_X,d_\Sigma)$ arise this way? 

\begin{question}
  Is the group of automorphisms of $\mathcal F(\Sigma_X,d_\Sigma)$ isomorphic to the group of quasi-isometric marked homeomorphisms of $(\Sigma,d_\Sigma, X)$?
\end{question}
The classical results show that this is the case if $\Sigma$ is compact (and hence $X$ is finite and the quasi-isometry requirements are vacuous).

Other natural questions arise \emph{e.g.}\ when considering the quotient of the flip-graph by its group of automorphisms as  in \cite{disarlo2018simultaneous,disarlo2019geometry}.

\subsection{Higher dimensions}\label{ssec: hi dim} 

Our proofs rely heavily on the fact that $\Sigma$ is 2-dimensional. We would be interested to see extensions of the results of this paper, especially \Cref{thm:triangulationIntro,thm: uni net intro}, to higher-dimensional manifolds:
\begin{question} \label{Q:triangulationIntro}
Let $(M,d_M)$ be a complete Riemannian manifold of dimension $n$. Must there exist a triangulation $\mathcal D$ of $(M,d_M)$ such that the identity map from the 1-skeleton $(G,d^G_\ell)$ of $\mathcal D$ to $(M,d_M)$  is a quasi-isometry? In case $(M,d_M)$ admits a uniform net, can we choose $\mathcal D$ so that the simplicial metric on $G$ is quasi-isometric to $(M,d_M)$?

Can we choose $\mathcal D$ so that, in addition to the above, its simplices have uniformly bounded intrinsic diameters? (we could restrict this to simplices of dimension $n$, or require it for all dimensions $m\leq n$). 
\end{question}

Bowditch \cite{BowBil} obtained results of this flavour under stronger assumptions on $(M,d_M)$, namely bounded curvature and injectivity radius, with the stronger conclusion of bi-Lipschitz embedding of the triangulation into $M$; see also \cites{dyer2015riemannian,boissonnat2018delaunay,saucan2005note} for related studies.

\bibliography{collective,morebib}

\begin{bibdiv}
\begin{biblist}

\bib{AhlSarRie}{book}{
      author={Ahlfors, Lars~V.},
      author={Sario, Leo},
       title={Riemann {Surfaces}},
   publisher={Princeton University Press},
        date={1960},
        ISBN={978-0-691-08027-7},
         url={https://www.jstor.org/stable/j.ctt183ph39},
}

\bib{bar2023big}{article}{
      author={Bar-Natan, Assaf},
      author={Goel, Advay},
      author={Halstead, Brendan},
      author={Hamrick, Paul},
      author={Shenoy, Sumedh},
      author={Verma, Rishi},
       title={Big flip graphs and their automorphism groups},
        date={2023},
     journal={Glasnik matemati{\v{c}}ki},
      volume={58},
      number={1},
       pages={125\ndash 134},
}

\bib{BisRemNon}{article}{
      author={Bishop, C.J.},
      author={Rempe, L.},
       title={{Non-compact Riemann surfaces are equilaterally triangulable}},
        date={2026},
     journal={Invent.\ math.},
      volume={244},
       pages={1\ndash 43},
}

\bib{boissonnat2018delaunay}{article}{
      author={Boissonnat, Jean-Daniel},
      author={Dyer, Ramsay},
      author={Ghosh, Arijit},
       title={Delaunay triangulation of manifolds},
        date={2018},
     journal={Foundations of Computational Mathematics},
      volume={18},
      number={2},
       pages={399\ndash 431},
}

\bib{BowBil}{unpublished}{
      author={Bowditch, B.~H.},
       title={{Bilipschitz triangulations of Riemannian manifolds}},
        note={Preprint 2020, {http://bhbowditch.com/papers/triangulations.pdf}},
}

\bib{BuKlSep}{article}{
      author={Burago, D.},
      author={Kleiner, B.},
       title={{Separated nets in Euclidean space and Jacobians of bi-Lipschitz maps}},
        date={1998},
     journal={GAFA},
      volume={8},
       pages={273\ndash 282},
}

\bib{DavStr}{article}{
      author={Davies, J.},
       title={String graphs are quasi-isometric to planar graphs},
        note={arXiv:2510.19602},
}

\bib{DavSur}{article}{
      author={Davies, J.},
       title={Surfaces without quasi-isometric simplicial triangulations},
        note={arXiv:2603.26652,},
}

\bib{disarlo2018simultaneous}{article}{
      author={Disarlo, V.},
      author={Parlier, H.},
       title={Simultaneous flips on triangulated surfaces},
        date={2018},
     journal={Michigan Math.\ J.},
      volume={67},
      number={3},
       pages={451\ndash 464},
}

\bib{disarlo2019geometry}{article}{
      author={Disarlo, Valentina},
      author={Parlier, Hugo},
       title={The geometry of flip graphs and mapping class groups},
        date={2019},
     journal={Trans.\ Amer.\ Math.\ Soc.},
      volume={372},
      number={6},
       pages={3809\ndash 3844},
}

\bib{doCarmo}{book}{
      author={do~Carmo, Manfredo~Perdigao},
       title={Differential geometry of curves and surfaces},
   publisher={Prentice-Hall, Inc},
        date={1976},
}

\bib{dyer2015riemannian}{article}{
      author={Dyer, Ramsay},
      author={Vegter, Gert},
      author={Wintraecken, Mathijs},
       title={Riemannian simplices and triangulations},
        date={2015},
     journal={Geometriae Dedicata},
      volume={179},
      number={1},
       pages={91\ndash 138},
}

\bib{fossas2022flip}{article}{
      author={Fossas, A.},
      author={Parlier, H.},
       title={Flip graphs for infinite type surfaces},
        date={2022},
     journal={Groups Geom.\ Dyn.},
      volume={16},
      number={4},
       pages={1165\ndash 1178},
}

\bib{GeoPapMin}{article}{
      author={Georgakopoulos, A.},
      author={Papasoglu, P.},
       title={{Graph minors and metric spaces}},
        date={2025},
     journal={Combinatorica},
      volume={45},
       pages={33},
}

\bib{guillemin2025differential}{book}{
      author={Guillemin, Victor},
      author={Pollack, Alan},
       title={Differential topology},
   publisher={American Mathematical Society},
        date={2025},
      volume={370},
}

\bib{hatcher}{book}{
      author={Hatcher, Allen},
       title={Algebraic {T}opology},
   publisher={Cambrigde Univ.\ Press},
        date={2002},
}

\bib{kahn2012counting}{article}{
      author={Kahn, Jeremy},
      author={Markovi{\'c}, Vladimir},
       title={Counting essential surfaces in a closed hyperbolic three-manifold},
        date={2012},
     journal={Geometry \& Topology},
      volume={16},
      number={1},
       pages={601\ndash 624},
}

\bib{KanaiRough2}{article}{
      author={Kanai, M.},
       title={{Rough isometries, and combinatorial approximations of geometries of non-compact Riemannian manifolds.}},
        date={1985},
     journal={J. Math. Soc. Japan},
      volume={37},
       pages={391\ndash 413},
}

\bib{korkmaz2012ideal}{inproceedings}{
      author={Korkmaz, M.},
      author={Papadopoulos, A.},
       title={On the ideal triangulation graph of a punctured surface},
        date={2012},
   booktitle={Annales de l'institut fourier},
      volume={62},
       pages={1367\ndash 1382},
}

\bib{Maillot}{article}{
      author={Maillot, S.},
       title={Quasi-isometries of groups, graphs and surfaces},
        date={2001},
     journal={Comment.\ Math.\ Helv.},
      volume={76},
      number={1},
       pages={29\ndash 60},
}

\bib{NgScSeAsyII}{unpublished}{
      author={Nguyen, T.},
      author={Scott, A.},
      author={Seymour, P.},
       title={{Asymptotic structure. II. Path-width and additive quasi-isometry}},
        note={ArXiv:2509.09031},
}

\bib{ntalampekos2023polyhedral}{article}{
      author={Ntalampekos, Dimitrios},
      author={Romney, Matthew},
       title={Polyhedral approximation of metric surfaces and applications to uniformization},
        date={2023},
     journal={Duke Mathematical Journal},
      volume={172},
      number={9},
       pages={1673\ndash 1734},
}

\bib{saucan2005note}{article}{
      author={Saucan, Emil},
       title={{Note on a theorem of Munkres}},
        date={2005},
     journal={Mediterr.\ J.\ Math.},
      volume={2},
      number={2},
       pages={215\ndash 229},
}

\end{biblist}
\end{bibdiv}

\end{document}